\documentclass[12pt]{amsart}
\usepackage{amscd,amsmath,amsthm,amssymb}
\usepackage{amsfonts,amssymb,amscd,amsmath,enumerate,verbatim}
\usepackage[left]{lineno}
\usepackage{pstricks, pst-plot,pst-3d}

\newpsstyle{fatline}{linewidth=1.5pt}
\newpsstyle{fyp}{fillstyle=solid,fillcolor=verylight}
\definecolor{verylight}{gray}{0.97}
\definecolor{light}{gray}{0.9}
\definecolor{medium}{gray}{0.85}
\definecolor{dark}{gray}{0.6}
\def\NZQ{\mathbb}               

\def\ZZ{{\NZQ Z}}
\def\RR{{\NZQ R}}

\def\frk{\mathfrak}               

\def\cc{{\frk c}}

\def\mm{{\frk m}}

\def\Dc{{\mathcal D}}
\def\Pc{{\mathcal P}}
\def\Qc{{\mathcal Q}}
\def\Mc{{\mathcal M}}
\def\Hc{{\mathcal H}}

\def\Bc{{\mathcal B}}

\def\Sc{{\mathcal S}}

\def\ab{{\mathbf a}}
\def\bb{{\mathbf b}}
\def\xb{{\mathbf x}}

\def\db{{\mathbf d}}
\def\eb{{\mathbf e}}

\def\opn#1#2{\def#1{\operatorname{#2}}} 
\opn\chara{char} \opn\length{\ell} \opn\pd{pd} \opn\rk{rk}
\opn\projdim{proj\,dim} \opn\injdim{inj\,dim} \opn\rank{rank}
\opn\depth{depth} \opn\grade{grade} \opn\height{height}
\opn\embdim{emb\,dim} \opn\codim{codim}
\opn\Cl{Cl}

\opn\Tr{Tr} \opn\bigrank{big\,rank}
\opn\superheight{superheight}\opn\lcm{lcm}
\opn\trdeg{tr\,deg}
	\opn\reg{reg} \opn\lreg{lreg} \opn\ini{in} \opn\lpd{lpd}
	\opn\size{size} \opn\sdepth{sdepth}
	\opn\link{link}\opn\fdepth{fdepth}\opn\lex{lex}
	\opn\tr{tr}
	\opn\type{type}
	\opn\gap{gap}
	\opn\arithdeg{arith-deg}
	\opn\revlex{revlex}
	\opn\div{div} \opn\Div{Div} \opn\cl{cl} \opn\Cl{Cl}
	\opn\Spec{Spec} \opn\Supp{Supp} \opn\supp{supp} \opn\Sing{Sing}
	\opn\Ass{Ass} \opn\Min{Min}\opn\Mon{Mon}
	\opn\Ann{Ann} \opn\Rad{Rad} \opn\Soc{Soc}
	\opn\Im{Im} \opn\Ker{Ker} \opn\Coker{Coker} \opn\Am{Am}
	\opn\Hom{Hom} \opn\Tor{Tor} \opn\Ext{Ext} \opn\End{End}
	\opn\Aut{Aut} \opn\id{id}
	
	\opn\nat{nat}
	\opn\pff{pf}
	\opn\Pf{Pf} \opn\GL{GL} \opn\SL{SL} \opn\mod{mod} \opn\ord{ord}
	\opn\Gin{Gin} \opn\Hilb{Hilb}\opn\sort{sort}
	\opn\PF{PF}\opn\Ap{Ap}
	\opn\mult{mult}
	\opn\bight{bight}
	\opn\div{div}
	\opn\Div{Div}
	\opn\aff{aff}
	\opn\relint{relint} \opn\st{st}
	\opn\lk{lk} \opn\cn{cn} \opn\core{core} \opn\vol{vol}  \opn\inp{inp} \opn\nilpot{nilpot}
	\opn\link{link} \opn\star{star}\opn\lex{lex}\opn\set{set}
	\opn\width{wd}
	\opn\Fr{F}
	\opn\QF{QF}
	\opn\G{G}
	\opn\type{type}\opn\res{res}
	\opn\conv{conv}
	\opn\Deg{Deg}
	\opn\Sym{Sym}
	\opn\Con{Con}
	\opn\gr{gr}
	
	\def\pot#1#2{#1[\kern-0.28ex[#2]\kern-0.28ex]}

	\opn\dirlim{\underrightarrow{\lim}}
	\opn\inivlim{\underleftarrow{\lim}}
	\let\to=\rightarrow
	
	\def\Implies{\ifmmode\Longrightarrow \else
		\unskip${}\Longrightarrow{}$\ignorespaces\fi}
	\def\implies{\ifmmode\Rightarrow \else
		\unskip${}\Rightarrow{}$\ignorespaces\fi}
	\def\iff{\ifmmode\Longleftrightarrow \else
		\unskip${}\Longleftrightarrow{}$\ignorespaces\fi}

	\let\:=\colon
	\newtheorem{Theorem}{Theorem}[section]
	\newtheorem{Lemma}[Theorem]{Lemma}
	\newtheorem{Corollary}[Theorem]{Corollary}

	\newtheorem{Example}[Theorem]{Example}

	\newtheorem{Conjecture}[Theorem]{Conjecture}

	\let\epsilon\varepsilon
	\let\kappa=\varkappa
	\def\qed{\ifhmode\textqed\fi
		\ifmmode\ifinner\quad\qedsymbol\else\dispqed\fi\fi}
	\def\textqed{\unskip\nobreak\penalty50
		\hskip2em\hbox{}\nobreak\hfil\qedsymbol
		\parfillskip=0pt \finalhyphendemerits=0}
	\def\dispqed{\rlap{\qquad\qedsymbol}}
	
	\opn\dis{dis}
	\def\pnt{{\raise0.5mm\hbox{\large\bf.}}}
	
	\opn\Lex{Lex}

\begin{document}
\title[Nearly Gorenstein polytopes]{Bounded powers of edge ideals: nearly Gorenstein polytopes}

\author[T.~Hibi]{Takayuki Hibi}
\author[S.~A.~ Seyed Fakhari]{Seyed Amin Seyed Fakhari}

\address{(Takayuki Hibi) Department of Pure and Applied Mathematics, Graduate School of Information Science and Technology, The University of Osaka, Suita, 
Osaka 565--0871, Japan}
\email{hibi@math.sci.osaka-u.ac.jp}
\address{(Seyed Amin Seyed Fakhari) Departamento de Matem\'aticas, Universidad de los Andes, Bogot\'a, Colombia}
\email{s.seyedfakhari@uniandes.edu.co}



\begin{abstract}
A lattice polytope $\Pc \subset \RR^n$ of dimension $n$ is called nearly Gorenstein if $\Pc$ is normal and if its Ehrhart ring is a nearly Gorenstein ring.  We study nearly Gorenstein polytopes arising from discrete polymatroids of bounded powers of edge ideals. Especially, it is shown that every nearly Gorenstein polytope coming from either a complete graph or a complete bipartite graph is level.       
\end{abstract}	
\maketitle
\thispagestyle{empty}

\section*{Introduction}
A {\em lattice polytope} $\Pc \subset \RR^n$ is a convex polytope each of whose vertices belongs to $\ZZ^n$.  A lattice polytope $\Pc \subset \RR^n$ of dimension $n$ is called {\em normal} if, for each $N = 2,3, \ldots$ and for each $\ab \in N \Pc \cap \ZZ^n$, there are $\ab^{(1)}, \ldots, \ab^{(N)}$ belonging to $\Pc \cap \ZZ^n$ for which $\ab = \ab^{(1)} + \cdots + \ab^{(N)}$, where $N\Pc = \{N\ab : \ab \in \Pc\}$.  Let $T=K[x_1^\pm, \ldots, x_n^\pm, t]$ denote the Laurent polynomial ring in $n+1$ variables over a field $K$.  If $\ab = (a_1, \ldots, a_n) \in \ZZ^n$, then set $\xb^\ab t:=x_1^{a_1} \cdots x_n^{a_n} t \in T$.  Given a lattice polytope $\Pc \subset \RR^n$ of dimension $n$, we introduce the subring $K[\Pc]$ of $T$ generated by those Laurent monomials $\xb^\ab t$ with $\ab \in \Pc \cap \ZZ^n$.  When $\Pc$ is normal, we say that $K[\Pc]$ is the {\em Ehrhart ring} of $\Pc$.  Every Ehrhart ring is a homogeneous algebra over $K$ and is Cohen--Macaulay.  Consult \cite{HibiRedBook} for further information about Ehrhart rings.

Let $\Pc \subset \RR^n$ be a normal polytope of dimension $n$ and $\partial \Pc$ the boundary of $\Pc$.  Set
\begin{eqnarray*}
\mm &:=& (\{\xb^\ab t : \ab \in \Pc \cap \ZZ^n\}) \subset K[\Pc], \\ 
\omega_{K[\Pc]} &:=& (\{\xb^\ab t^N : \ab \in N(\Pc \setminus \partial \Pc) \cap \ZZ^n\}) \subset K[\Pc].
\end{eqnarray*}
In other words, $\mm$ is the maximal ideal of $K[\Pc]$ and $\omega_{K[\Pc]}$ is the canonical module of $K[\Pc]$.  Let $a(K[\Pc])$ denote the $a$-invariant of $K[\Pc]$.  One has $a(K[\Pc]) = - s$, where    
\[
s := \min\{N \geq 1 : N(\Pc \setminus \partial \Pc) \cap \ZZ^n\} \neq \emptyset\}.
\]
We say that $K[\Pc]$ is {\em level} \cite{Stanley} if
\[
\omega_{K[\Pc]} = (\{\xb^\ab t^s : \ab \in s(\Pc \setminus \partial \Pc) \cap \ZZ^n\})
\]
and that $K[\Pc]$ is pseudo-Gorenstein \cite{pseudo} if $|s(\Pc \setminus \partial \Pc) \cap \ZZ^n|=1$.  Recall that $K[\Pc]$ is Gorenstein if $\omega_{K[\Pc]}$ is generated by exactly one element.  In other words, $K[\Pc]$ is Gorenstein if and only if $K[\Pc]$ is level and pseudo-Gorenstein.  On the other hand, $K[\Pc]$ is {\em nearly Gorenstein} \cite{nearlyGor} if
\[
\mm \subset \omega_{K[\Pc]} \cdot \omega_{K[\Pc]}^{-1}, 
\]
where
\[
\omega_{K[\Pc]}^{-1} = \{ f \in Q(K[\Pc]): f \cdot \omega_{K[\Pc]} \subset K[\Pc] \},
\]
and where $Q(K[\Pc])$ is the quotient field of $K[\Pc]$. 
If $K[\Pc]$ is Gorenstein, then $\omega_{K[\Pc]} \cdot \omega_{K[\Pc]}^{-1} = K[\Pc]$ and $K[\Pc]$ is nearly Gorenstein. It follows from \cite{Miyashita} that if $K[\Pc]$ is pseudo-Gorenstein and nearly Gorenstein, then $K[\Pc]$ is Gorenstein.

Let $S=K[x_1, \ldots,x_n]$ denote the polynomial ring  in $n$ variables over a field $K$ with $n \geq 3$.  If $u \in S$ is a monomial, then $M_{\leq u}$ stands for the set of those monomials $w \in S$ which divide $u$.  In particular, $1 \in M_{\leq u}$ and $u \in M_{\leq u}$.  Let $G$ be a finite graph on the vertex set $V(G)=\{x_1, \ldots, x_n\}$, where $n \geq 3$, with no loop, no multiple edge and no isolated vertex, and $E(G)$ the set of edges of $G$.  The edge ideal of $G$ is the ideal $I(G) \subset S$ which is generated by those $x_ix_j$ with $\{x_i, x_j\} \in E(G)$. 

Let $\ZZ_{>0}$ denote the set of positive integers. Given $\mathfrak{c}=(c_1,\ldots,c_n)\in (\ZZ_{>0})^n$ and an integer $q\geq 1$, we denote by $(I(G)^q)_\mathfrak{c}$ the ideal of $S$ generated by those monomials $x_1^{a_1}\cdots x_n^{a_n} \in I(G)^q$ with each $a_i \leq c_i$.  Let $\delta_{\mathfrak{c}}(I(G))$ denote the biggest integer $q$ for which $(I(G)^q)_\mathfrak{c} \neq (0)$.  Then $(I(G)^{\delta_{\mathfrak{c}}(I(G))})_\mathfrak{c}$ is a polymatroidal ideal \cite[Theorem 4.3]{HSF_1}.  

Let $\Bc(G,\mathfrak{c})$ denote the minimal set of monomial generators of $(I(G)^{\delta_{\mathfrak{c}}(I(G))})_\mathfrak{c}$.  Set $\Mc(G,\mathfrak{c}):= \bigcup_{u \in \Bc(G,\mathfrak{c})}M_{ \leq u}$ and 
\[
\Dc(G,\mathfrak{c}):=\{(a_1, \ldots, a_n) \in \ZZ^n : x_1^{a_1}\cdots x_n^{a_n} \in \Mc(G,\mathfrak{c})\}.
\]
The standard unit coordinate vectors $\eb_1, \ldots, \eb_n$ of $\RR^n$ together with the origin $(0,\ldots,0)$ of $\RR^n$ belong to $\Dc(G,\mathfrak{c})$.  Now, since $(I(G)^{\delta_{\mathfrak{c}}(I(G))})_\mathfrak{c}$ is a polymatroidal ideal, it follows from \cite[Theorem 2.3]{HH_discrete} that $\Dc(G,\mathfrak{c})$ is a discrete polymatroid \cite[Definition 2.1]{HH_discrete}.  Let $\conv(\Dc(G,\mathfrak{c})) $ denote the convex hull of $\Dc(G,\mathfrak{c})$ in $\RR^n$.  It then follows from \cite[Theorem 3.4]{HH_discrete} that $\conv(\Dc(G,\mathfrak{c}))$ is a polymatroid \cite[Definition 1.1]{HH_discrete}.  In particular, $\conv(\Dc(G,\mathfrak{c}))$ is normal \cite[Theorem 6.1]{HH_discrete}. 

Let $2^{[n]}$ denote the set of subsets of $[n]:=\{1, \ldots, n\}$.  The {\em ground set rank function} \cite[p.~243]{HH_discrete} $\rho_{(G,\mathfrak{c})} : 2^{[n]} \to \ZZ_{>0}$ of $\conv(\Dc(G,\mathfrak{c}))$ is defined by setting
\[
\rho_{(G,\mathfrak{c})}(X) = \max\left\{\sum_{i \in X}a_i : x_1^{a_1}\cdots x_n^{a_n} \in \Bc(G,\mathfrak{c})\right\}
\]
for $\emptyset \neq X \subset [n]$ together with $\rho_{(G,\mathfrak{c})}(\emptyset)=0$. A nonempty subset $A\subset [n]$ is called {\em $\rho_{(G,\mathfrak{c})}$-closed} if for any subset $B\subset [n]$ with $A \subsetneq B$, one has $\rho_{(G,\mathfrak{c})}(A) < \rho_{(G,\mathfrak{c})}(B)$. A nonempty subset $A\subset [n]$ is called {\em $\rho_{(G,\mathfrak{c})}$-separable} if there exist nonempty subsets $A'$ and $A''$ of $[n]$ with $A = A' \cup A''$ and $A' \cap A'' = \emptyset$ for which $\rho_{(G,\mathfrak{c})}(A) = \rho_{(G,\mathfrak{c})}(A') + \rho_{(G,\mathfrak{c})}(A'')$. We say that a subset $A\subset [n]$ is {\em $\rho_{(G,\mathfrak{c})}$-inseparable} if it is not $\rho_{(G,\mathfrak{c})}$-separable.  Given $\emptyset \neq A \subset [n]$, we define the hyperplane $\Hc_A \subset \RR^n$ by
\[
\Hc_A = \{ (x_1,\ldots,x_n) \in \RR^n : \sum_{i\in A} x_i = \rho(A) \}.
\]  
Given $i\in [n]$, we define the hyperplane $\Hc^{(i)}  \subset \RR^n$ by
\[
\Hc^{(i)} = \{ (x_1,\ldots,x_n) \in \RR^n : x_i = 0 \}.
\] 
Edmonds \cite{Edmonds} says that
 
\begin{Lemma} \label{facets}
The facets of $\conv(\Dc(G,\mathfrak{c})) \subset \RR^n$ are all $\Hc^{(i)} \cap \conv(\Dc(G,\mathfrak{c}))$ with $i \in [n]$ and all $\Hc_A \cap \conv(\Dc(G,\mathfrak{c}))$, where $A$ ranges over all $\rho_{(G,\mathfrak{c})}$-closed and $\rho_{(G,\mathfrak{c})}$-inseparable subsets of $[n]$. 
\end{Lemma}

 Let $G$ be a finite graph on $V(G) = \{x_1, \ldots, x_n\}$.  A lattice polytope $\Pc \subset \RR^n$ of dimension $n$ is called a {\em $G$-polytope} if there is $\cc \in (\ZZ_{>0})^n$ with $\Pc = \conv(\Dc(G,\mathfrak{c}))$.

In general, there does exist a nearly Gorenstein ring which is not level \cite{level}.  However, it can be conjectured that every nearly Gorenstein normal homogeneous toric ring is level.  In particular, 

\begin{Conjecture}
\label{conjecture}
Every nearly Gorenstein $G$-polytope is level.
\end{Conjecture}

 In the present paper, it is shown that Conjecture \ref{conjecture} is true for complete graphs (Corollary \ref{NearLevel}) and for complete bipartite graphs (Corollary \ref{final}).

\section{Complete graphs}
Let $n \geq 3$, and let $K_n$ denote the complete graph on $V(G) = \{x_1, \ldots, x_n\}$.  Nearly Gorenstein $K_n$-polytopes as well as level $K_n$-polytopes are classified. 

Given integers $k_1, \ldots, k_n$ with each $k_i \geq 1$, the cube in $\RR^n$ defined by the system of linear inequalities $0 \leq x_i \leq k_i, 1 \leq i \leq n$, is denoted by $\Qc_{k_1,\ldots, k_n}$.

\begin{Lemma}
\label{cube}
Let $n \geq 3$.  The cube $\Qc_{k_1,\ldots, k_n} \subset \RR^n$ is a $K_n$-polytope is and only if $k_1+\cdots +k_n$ is even and $k_i\leq \sum_{1 \leq j \leq n, \, j \neq i} k_j$ for each $1 \leq i \leq n$.
\end{Lemma}

\begin{proof}
The cube $\Qc_{k_1,\ldots, k_n} \subset \RR^n$ is a $K_n$-polytope if and only if there is $\cc \in (\ZZ_{>0})^n$ for which $\Bc(K_n,\cc) = \{x_1^{k_1}\cdots x_n^{k_n}\}$.  Supposing that $\Bc(K_n,\cc) = \{x_1^{k_1}\cdots x_n^{k_n}\}$, then the monomial $x_1^{k_1}\cdots x_n^{k_n}$ can be written as the product of certain number of quadratic squarefree monomials of the form $x_ix_j$, $1\leq i<j\leq n$. Hence, $k_1+\cdots +k_n$ is even and $k_i\leq \sum_{1 \leq j \leq n, \, j \neq i} k_j$ for each $1 \leq i \leq n$.  Conversely, suppose that $k_1+\cdots +k_n$ is even and $k_i\leq \sum_{1 \leq j \leq n, \, j \neq i} k_j$ for each $1 \leq i \leq n$. Set $\cc:=(k_1, \ldots, k_n)$ and $\delta:=\delta_{\cc}(I(K_n))$. Choose a monomial $u=x_1^{a_1}\cdots x_n^{a_n}\in\Bc(K_n,\cc)$. We show that $u=x_1^{k_1}\cdots x_n^{k_n}$. If $a_i=k_i$, for each $1\leq i\leq n$, then we are done. So, let $a_i< k_i$, for some integer $i$ with $1\leq i\leq n$. If there is an integer $j\neq i$ with $1\leq j\leq n$ for which $a_j< k_j$, then $(x_ix_j)u\in (I(G)^{\delta+1})_{\cc}$, contradicting the definition of $\delta$. Thus $a_j=k_j$, for each $1\leq j\leq n$ with $j\neq i$. Since $\deg(u)$ and $k_1+\cdots +k_n$ are even integers, we deduce that $a_i\leq k_i-2$. It follows from $u\in\Bc(K_n,\cc)$ that $u=u_1u_2\cdots u_{\delta}$, where each $u_{\ell}$ is a quadratic squarefree monomial. If every $u_{\ell}$ is divisible by $x_i$, then$$k_i\geq a_i+2=\sum_{1 \leq j \leq n, \, j \neq i} a_j+2=\sum_{1 \leq j \leq n, \, j \neq i} k_j+2,$$contradicting our assumption. Hence, at least one of $u_{\ell}$'s, say, $u_1$ is not divisible by $x_i$. Let $u_1=x_px_q$ with $p,q\neq i$. Then$$x_i^2u=(x_ix_p)(x_ix_q)u_2\ldots u_{\delta}\in (I(G)^{\delta+1})_{\cc},$$contradicting the definition of $\delta$.
\end{proof}

\begin{Lemma}
    \label{cube2}
    Let $n \geq 3$. 
    \begin{itemize}
        \item[(1)] 
        The cube $\Qc_{k_1,\ldots, k_n} \subset \RR^n$ is Gorenstein if and only if either $k_1 = \cdots = k_n = 2$, or $k_1 = \cdots = k_n = 1$.
        \item[(2)] 
        The cube $\Qc_{k_1,\ldots, k_n} \subset \RR^n$ is nearly Gorenstein if and only if either each $k_i \leq 2$ or each $k_i \geq 2$. 
        
        \item[(3)]
        The cube $\Qc_{k_1,\ldots, k_n} \subset \RR^n$ is level for all $k_1 \geq 1, \ldots, k_n\geq 1$.
    \end{itemize}
\end{Lemma}

\begin{proof}
\noindent
{\bf (Gorenstein (1))} See \cite[Theorem 4.5]{HSF_5}.  

\medskip
\noindent
{\bf (nearly Gorenstein (2))} Set $\Qc:=\Qc_{k_1,\ldots, k_n}$ and ${\bf x}_A:=\prod_{i\in A}x_i$ for each $A\subset [n]$.   

Suppose that each $k_i \leq 2$.
If either $k_1=\cdots =k_n=1$, or $k_1=\cdots =k_n=2$, then $\Qc$ is Gorenstein. So, one assumes that there is an integer $1\leq r\leq n-1$ for which $k_1=\cdots =k_r=1$ and $k_{r+1}=\cdots =k_n=2$. Suppose that $\ab=(a_1, \ldots, a_n)\in \Qc\cap \ZZ^{n}$.  We show that $\xb^\ab t\in \omega_{K[\Qc]}.\omega_{K[\Qc]}^{-1}$. Set $A:=\{i\in [n] : a_i=0\}\cup\{r+1, \ldots, n\}$.  Then $(\xb_{A}t)(\xb^\ab t)\in \omega_{K[\Qc]}$. Since $1/(\xb_{A}t)\in\omega_{K[\Qc]}^{-1}$, the assertion follows.

Assume that each $k_i \geq 2$. As above, suppose that $\ab=(a_1, \ldots, a_n)\in \Qc\cap \ZZ^{n}$.  We show that $\xb^\ab t\in \omega_{K[\Qc]}.\omega_{K[\Qc]}^{-1}$. Set $A_1:=\{i\in [n] : a_i=k_i\}$ and $A_2:=\{i\in [n] : a_i=0\}$. Then $A_1$ and $A_2$ are disjoint subsets of $[n]$.  One has $\xb_{A_2}/\xb_{A_1}\xb^\ab t\in \omega_{K[\Qc]}$. Since $\xb_{A_1}/\xb_{A_2}\in\omega_{K[\Qc]}^{-1}$, the assertion follows.

Finally, suppose that neither of the above situation holds. Without loss of generality, we may assume that  $k_1=1$ and $k_n\geq 3$. We show that $x_1^{k_1}\cdots x_n^{k_n}t\in K[\Qc]$ does not belong to $\omega_{K[\Qc]}.\omega_{K[\Qc]}^{-1}$. By contradiction, assume that there are $\xb^{\bb}t^{\ell}\in \omega_{K[\Qc]}$ and $\xb^{\bb'}t^{\ell'}\omega_{K[\Qc]}^{-1}$ with $$x_1^{k_1}\cdots x_{n}^{k_n}t=(\xb^{\bb}t^{\ell})(\xb^{\bb'}t^{\ell'})=\xb^{\bb+\bb'}t^{\ell+\ell'}.$$
One has $\bb+\bb'=(k_1, \ldots, k_n)$ and $\ell+\ell'=1$. Since $x_1\cdots x_nt^2\in \omega_{K[\Qc]}$ and since $\xb^{\bb'}t^{\ell'}\in \omega_{K[\Qc]}^{-1}$, it follows that
\begin{eqnarray}
\label{hhhhh}   
(x_1\cdots x_nt^2)(\xb^{\bb'}t^{\ell'})\in K[\Qc].
\end{eqnarray}
In particular, $\ell'\geq -2$. Let $\ell'=-2$. It follows from (\ref{hhhhh}) that $\xb^{\bb'}=x_1^{-1}\cdots x_n^{-1}$.  Since $x_1\cdots x_{n-1}x_n^2t^2\in\omega_{K[\Qc]}$, we easily see that
\begin{align*}
(x_1\cdots x_{n-1}x_n^2t^2)(\xb^{\bb'}t^{\ell'})=(x_1\cdots x_{n-1}x_n^2t^2)(x_1^{-1}\cdots x_n^{-1}t^{-2})=x_n  
\end{align*}
belongs to $K[\Qc]$, a contradiction.  Hence, $\ell'\geq -1$. Set $\bb'=(b_1', \ldots, b_n')$. It follows from (\ref{hhhhh}) that $b_i'\geq -1$ for each $i\in [n]$. Since $x_1\cdots x_{n-1}x_n^{2k_n-1}t^2\in\omega_{K[\Qc]}$, one has $$(x_1\cdots x_{n-1}x_n^{2k_n-1}t^2)(\xb^{\bb'}t^{\ell'})\in K[\Qc].$$As the exponent of $x_n$ in the above product is at least $2k_n-2>k_n$, we deduce that $\ell'\neq -1$. Thus $\ell'\geq 0$. On the other hand, Since $(\Qc \setminus \partial \Qc) \cap \ZZ^n = \emptyset$, one has $\ell\geq 2$, contradicting $\ell+\ell'=1$.

\medskip
\noindent
{\bf (level (3))}  Clearly, every cube is level.
\hspace{6.5cm}
\end{proof}

In the rest of this section, we identify any edge of $K_n$ with the corresponding quadratic squarefree monomial. For example, the edge $\{1,2\}\in E(K_n)$ will be identified with the monomial $x_1x_2$.

\begin{Lemma}
\label{cube3}
Let $n \geq 4$.  If $|\Bc(K_n,\mathfrak{c})|\geq 2$, then $\conv(\Dc(K_n,\mathfrak{c}))$ can be neither nearly Gorenstein nor level.   
\end{Lemma}

\begin{proof}
Let $|\Bc(K_n,\mathfrak{c})|\geq 2$.  Set $\delta:=\delta_{\cc}(I(K_n))$. If $c_i> (c_1+\cdots +c_n)-c_i$ for some $i\in [n]$, then replacing $c_i$ with $(c_1+\cdots +c_n)-c_i$ does not change $\Bc(K_n,\mathfrak{c})$. So, one can assume that each $c_i\leq (c_1+\cdots +c_n)-c_i$, for each $i\in [n]$. We claim $\delta=(c_1+\cdots +c_n-1)/2$. 

Clearly, $\delta\leq (c_1+\cdots +c_n)/2$. If $\delta =  (c_1+\cdots +c_n)/2$, then $\Bc(K_n,\mathfrak{c})=\{x_1^{c_1}\cdots x_n^{c_n}\}$ is a singleton, contradicting our assumption. Let $\delta\leq (c_1+\cdots +c_n-2)/2$.  We choose a monomial $u=x_1^{a_1}\cdots x_n^{a_n}\in \Bc(K_n,\mathfrak{c})$.  Since $u\in I(K_n)^{\delta}$, it can be written as $u=e_1\cdots e_{\delta}$, for (not necessarily distinct) edges $e_1, \ldots, e_{\delta}\in E(K_n)$. If there are $p, q\in [n]$, $p\neq  q$, with $a_p<c_p$ and $a_q< c_q$, then $(x_px_q)u\in (I(K_n)^{\delta+1})_{\cc}$, contradicting the definition of $\delta$. So, there is at most one $r\in [n]$ with $a_r <c_r$. Since$$\deg u=2\delta\leq (c_1+\cdots +c_n-2),$$we deduce that $a_r\leq c_r-2$. If in the presentation $u=e_1\cdots e_{\delta}$, there is an edge, say $e_1$, which is not incident to $x_r$, then $e_1=x_tx_{t'}$ for $t,t'\in [n]$, $t \neq t'$, with $r\notin\{t,t'\}$. Then$$x_r^2u=(x_rx_t)(x_rx_{t'})e_2\cdots e_{\delta}\in (I(K_n)^{\delta+1})_{\cc},$$contradicting the definition of $\delta$. Therefore, in the presentation $u=e_1\cdots e_{\delta}$, all edges $e_1, \ldots, e_{\delta}$ are incident to $x_r$. This implies that $a_r=\delta$. Hence,$$c_r\geq a_r+2=\delta+2=\sum_{\substack{i\in [n]\\ i\neq r}}a_i+2=\sum_{\substack{i\in [n]\\ i\neq r)}}c_i+2,$$contradicting $c_r\leq (c_1+\cdots +c_n)-c_r$. Thus $\delta=(c_1+\cdots +c_n-1)/2$, as desired.

Let $v\in \Bc(K_n,\mathfrak{c})$. As $v\in I(K_n)^{\delta}$, it can be written as $v=f_1\cdots f_{\delta}$, for (not necessarily distinct) edges $f_1, \ldots, f_{\delta}\in E(K_n)$. Since $\deg v=2\delta=c_1+\cdots +c_n-1$, there is $h\in [n]$ with $\deg_{x_h}v=c_h-1$ and $\deg_{x_j}v=c_j$ for each $j\in [n]\setminus \{h\}$. As above, one easily sees that at least one of the edges $f_1, \ldots, f_{\delta}$, say, $f_1$ is not incident to $x_h$. Therefore, $f_1=x_{h'}x_{h''}$ for distinct integers $h',h''\in [n]\setminus\{h\}$. Then$$x_hv/x_{h'}=(x_hx_{h''})f_2\cdots f_{\delta}\in  \Bc(K_n,\mathfrak{c}).$$This computation shows that $\rho_{(K_n,\mathfrak{c})}(X)=\sum_{i\in X}c_i$ for each nonempty proper subset $X$ of $[n]$.  Furthermore, $\rho_{(K_n,\mathfrak{c})}([n])=2\delta=c_1+\cdots +c_n-1$. As a consequence, a nonempty subset $X\subset [n]$ is $\rho_{(K_n,\mathfrak{c})}$-closed and $\rho_{(K_n,\mathfrak{c})}$-inseparable if and only if $X=[n]$ or $X$ is a singleton.  It then follows from Lemma \ref{facets} that $\conv(\Dc(K_n,\mathfrak{c}))$ is defined by the system of linear inequalities $0\leq x_i\leq c_i$, $1 \leq i \leq n$, together with $x_1+\cdots +x_n\leq c_1+\cdots +c_n-1$.

\medskip

\noindent
{\bf (First Step)}
We show that $\conv(\Dc(K_n,\mathfrak{c}))$ is not nearly Gorenstein.

Set $\Pc:=\conv(\Dc(K_n,\mathfrak{c}))$. Consider $x_1^{c_1}\cdots x_{n-1}^{c_{n-1}}x_n^{c_n-1}t\in K[\Pc]$. To show that $\Pc$ is not nearly Gorenstein, it is enough to prove that$$x_1^{c_1}\cdots x_{n-1}^{c_{n-1}}x_n^{c_n-1}t\notin \omega_{K[\Pc]}.\omega_{K[\Pc]}^{-1}.$$By contradiction, assume that there are $\xb^{\bb}t^{\ell}\in \omega_{K[\Pc]}$ and $\xb^{\bb'}t^{\ell'}\omega_{K[\Pc]}^{-1}$ for which $$x_1^{c_1}\cdots x_{n-1}^{c_{n-1}}x_n^{c_n-1}t=(\xb^{\bb}t^{\ell})(\xb^{\bb'}t^{\ell'})=\xb^{\bb+\bb'}t^{\ell+\ell'}.$$Therefore, $\bb+\bb'=(c_1, \ldots, c_{n-1}, c_n-1)$ and $\ell+\ell'=1$. Since $x_1\cdots x_nt^2\in \omega_{K[\Pc]}$ and $\xb^{\bb'}t^{\ell'}\in \omega_{K[\Pc]}^{-1}$, one has 
\begin{eqnarray}
\label{eeeeeE}   
(x_1\cdots x_nt^2)(\xb^{\bb'}t^{\ell'})\in K[\Pc].
\end{eqnarray}
In particular, $\ell'\geq -2$. Suppose that $\ell'=-2$. It then follows from (\ref{eeeeeE}) that $\xb^{\bb'}=x_1^{-1}\cdots x_n^{-1}$.  Since $n\geq 4$, one has $x_1^{3c_1-1}\cdots x_n^{3c_n-1}t^3\in\omega_{K[\Pc]}$.  However,
\begin{eqnarray*}
(x_1^{3c_1-1}\cdots x_n^{3c_n-1}t^3)(\xb^{\bb'}t^{\ell'})&=&(x_1^{3c_1-1}\cdots x_n^{3c_n-1}t^3)(x_1^{-1}\cdots x_n^{-1}t^{-2})\\ &=&x_1^{3c_1-2}\cdots x_n^{3c_n-2}t  
\end{eqnarray*}
does not belong to $K[\Pc]$, as$$(3c_1-2)+\cdots +(3c_n-2)=3(c_1+\cdots +c_n)-2n> c_1+\cdots +c_n-1.$$
This is a contradiction. Hence, $\ell'\geq -1$.

Set $\bb:=(b_1, \ldots, b_n)$ and $\bb':=(b_1', \ldots, b_n')$. As $\ell\geq 1$, the equality $\ell+\ell'=1$ implies that either $(\ell,\ell')=(1,0)$ or $(\ell,\ell')=(2,-1)$.

\medskip

\noindent
{\bf (Case 1)} Let $(\ell,\ell')=(1,0)$.  Since  $\xb^{\bb}t=\xb^{\bb}t^{\ell}\in \omega_{K[\Pc]}$, one has $b_i\leq c_i-1$, for each $i\in [n]$.  We deduce from $\bb+\bb'=(c_1, \ldots, c_{n-1}, c_n-1)$ that $b_i'\geq 1$ for every $i\in[n-1]$, and $b_n'\geq 0$. For each $i\in [n]$, define $b_i''=(n-1)c_i-1$ and set $\bb'':=(b_1'', \ldots, b_n'')$.  The defining inequalities of $\Pc$ says that $\bb''\in (n-1)(\Pc\setminus \partial\Pc)$. In other words, $\xb^{\bb''}t^{n-1}\in \omega_{K[\Pc]}$. Thus,$$\xb^{\bb'+\bb''}t^{n-1}=(\xb^{\bb'})(\xb^{\bb''}t^{n-1})=(\xb^{\bb'}t^{\ell'})(\xb^{\bb''}t^{n-1})\in K[\Pc].$$This yields that $\bb'+\bb''\in (n-1)\Pc$. However, this is impossible as
\begin{eqnarray*}
&&(b_1'+\cdots +b_n')+(b_1''+\cdots +b_n'') \\&\geq& (n-1)+(n-1)(c_1+\cdots +c_n)-n\\ &>& (n-1)(c_1+\cdots +c_n-1). 
\end{eqnarray*}

\medskip

\noindent
{\bf (Case 2)} Let $(\ell,\ell')=(2,-1)$.  Since $\xb^{\bb}t^2=\xb^{\bb}t^{\ell}\in \omega_{K[\Pc]}$, one has $\bb\in 2(\Pc\setminus\partial\Pc)$. Hence $b_i\leq 2c_i-1$ for each $i\in [n]$. This implies that$$b_1+\cdots +b_n\leq 2(c_1+\cdots +c_n)-n.$$We conclude from $\bb+\bb'=(c_1, \ldots, c_{n-1}, c_n-1)$ and the above inequality that $$b_1'+\cdots +b_n'\geq -(c_1+\cdots +c_n)+n-1.$$Let $\bb''$ be the same vector defined in Case 1. Since $\xb^{\bb''}t^{n-1}\in \omega_{K[\Pc]}$, one has  $$\xb^{\bb'+\bb''}t^{n-2}=(\xb^{\bb'}t^{-1})(\xb^{\bb''}t^{n-1})=(\xb^{\bb'}t^{\ell'})(\xb^{\bb''}t^{n-1})\in K[\Pc].$$This yields that $\bb'+\bb''\in (n-2)\Pc$. However, this is impossible as
\begin{eqnarray*}
&&(b_1'+\cdots +b_n')+(b_1''+\cdots +b_n'')\\ &\geq& -(c_1+\cdots +c_n)+(n-1)+(n-1)(c_1+\cdots +c_n)-n\\ &=&(n-2)(c_1+\cdots +c_n)-1>(n-2)(c_1+\cdots +c_n-1),
\end{eqnarray*}
where the last inequality follows from $n\geq 4$.

\medskip

\noindent
{\bf (Second Step)}
We show that $\conv(\Dc(K_n,\mathfrak{c}))$ is not level. 

Note that $(3c_1-1, \ldots, 3c_n-1)\in \ZZ_{>0}^n$ belongs to the interior of  $3\conv(\Dc(K_n,\mathfrak{c}))$. But it cannot be written as the sum of a point belonging to $\conv(\Dc(K_n,\mathfrak{c})) \cap \ZZ^n$ and a point belonging to $2(\conv(\Dc(K_n,\mathfrak{c})) \setminus \partial\conv(\Dc(K_n,\mathfrak{c}))) \cap \ZZ^n$.
\, \, \, \, \, \, \, \, \, \, \, \,
\end{proof}

\begin{Lemma} \label{n=3}
Let $\cc=(c_1, c_2, c_3)\in (\ZZ_{>0})^3$ and suppose that $|\Bc(K_3,\mathfrak{c})|\geq 2$. 
\begin{itemize}
\item[(1)] $\conv(\Dc(K_3,\mathfrak{c}))$ is Gorenstein if and only if $(c_1, c_2, c_3)=(1,1,1)$.
\item[(2)] $\conv(\Dc(K_3,\mathfrak{c}))$ is nearly Gorenstein if and only if $(c_1, c_2, c_3)=(1,1,1)$.
\item[(3)] $\conv(\Dc(K_3,\mathfrak{c}))$ is level if and only if $\min\{c_1, c_2, c_3\}=1$.
\end{itemize}
\end{Lemma}

\begin{proof}
As in the proof of Lemma \ref{cube3}, one may assume that $c_i\leq (c_1+c_2+c_3)-c_i$, for each $i=1,2,3$. Furthermore, the same argument as in the proof of Lemma \ref{cube3} shows that $\conv(\Dc(K_3,\mathfrak{c}))$ is defined by the system of linear inequalities $0\leq x_1\leq c_1, 0\leq x_2\leq c_2, 0\leq x_3\leq c_3$ together with $x_1+x_2+x_3\leq c_1+c_2+c_3-1$.  If $(c_1, c_2, c_3)=(1,1,1)$, then $\conv(\Dc(K_3,\mathfrak{c}))$ is Gorenstein (\cite[Theorem 4.5]{HSF_5}).  

\medskip

\noindent
{\bf (First Step)}
Let $\cc\neq (1,1,1)$ and $|\Bc(K_n,\mathfrak{c})|\geq 2$.  We show that $\conv(\Dc(K_3,\mathfrak{c}))$ is not nearly Gorenstein.  Since $\cc\neq (1,1,1)$, at least one component of $\cc$, say, $c_3$ is bigger than one. If $c_1=c_2=1$, then $\Bc(K_3,\mathfrak{c})=\{x_1x_2x_3^2\}$ is a singleton, contradicting our assumption. So, at least one of $c_1$ and $c_2$, say, $c_2$ is bigger than one. In other words, $c_2\geq 2$ and $c_3\geq 2$.

Set $\Pc:=\conv(\Dc(K_3,\mathfrak{c}))$. To show that $\Pc$ is not nearly Gorenstein, it is enough to prove that$$x_1^{c_1}x_2^{c_2} x_3^{c_3-1}t\notin \omega_{K[\Pc]}.\omega_{K[\Pc]}^{-1}.$$By contradiction, assume that there are $\xb^{\bb}t^{\ell}\in \omega_{K[\Pc]}$ and $\xb^{\bb'}t^{\ell'}\omega_{K[\Pc]}^{-1}$ for which $$x_1^{c_1}x_2^{c_2} x_3^{c_3-1}t=(\xb^{\bb}t^{\ell})(\xb^{\bb'}t^{\ell'})=\xb^{\bb+\bb'}t^{\ell+\ell'}.$$Therefore, $\bb+\bb'=(c_1, c_2, c_3-1)$ and $\ell+\ell'=1$. By the same argument as in the proof of Lemma \ref{cube3}, one has  $\ell'\geq -2$.

\medskip

\noindent
{\bf (Case 1)} Let $\ell'=-2$. Then the same argument as in the proof of Lemma \ref{cube3} shows that $\xb^{\bb'}=x_1^{-1}x_2^{-1}x_3^{-1}$. On the other hand, $x_1^{2c_1-1}x_2^{2c_2-1} x_3^{2c_3-1}t^2\in\omega_{K[\Pc]}$.  However,
\begin{eqnarray*}
& & (x_1^{2c_1-1}x_2^{2c_2-1} x_3^{2c_3-1}t^2)(\xb^{\bb'}t^{\ell'})\\
&=&  (x_1^{2c_1-1}x_2^{2c_2-1} x_3^{2c_3-1}t^2)(x_1^{-1}x_2^{-1}x_3^{-1}t^{-2})\\  &=&  x_1^{2c_1-2}x_2^{2c_2-2}x_3^{2c_3-2}  
\end{eqnarray*}
does not belong to $K[\Pc]$, as $c_2\geq 2$ and $c_3\geq 2$, a contradiction. 

\medskip

\noindent
{\bf (Case 2)} Let $\ell'=-1$. It follows from $\ell+\ell'=1$ that $\ell=2$. Let $\bb=(b_1,b_2,b_3)$ and $\bb'=(b_1',b_2',b_3')$. We deduce from $\xb^{\bb}t^2=\xb^{\bb}t^{\ell}\in \omega_{K[\Pc]}$ that $\bb\in 2(\Pc\setminus\partial\Pc)$. Hence, $b_i\leq 2c_i-1$ for each $i$. If $b_i\leq 2c_i-2$. for some $i$, then$$b_1+b_2+b_3\leq 2(c_1+c_2+c_3)-4,$$and the same argument as in Case 2 of the proof of Lemma \ref{cube3} yields a contradiction. Let $b_i=2c_i-1$, for each $i\in [3]$. It follows from $\bb+\bb'=(c_1, c_2, c_3-1)$ that $\bb'=(1-c_1, 1-c_2,-c_3)$. In particular, $b_3'\leq -2$, as $c_3\geq 2$. On the other hand, since $x_1x_2x_3t^2\in \omega_{K[\Pc]}$ and $\xb^{\bb'}t^{\ell'}\in \omega_{K[\Pc]}^{-1}$, one has
\begin{eqnarray*}   
(x_1\cdots x_nt^2)(\xb^{\bb'}t^{\ell'})\in K[\Pc].
\end{eqnarray*}
This implies that any component of $\bb'$ is at least $-1$, a contradiction.

\medskip

\noindent
{\bf (Case 3)} Let $\ell'\geq 0$. Since $\xb^{\bb}t^{\ell}\in \omega_{K[\Pc]}$, we have $\ell\geq 1$. It follows from $\ell+\ell'=1$ that $(\ell.\ell')=(1,0)$. The same argument as in Case 1 of the proof of Lemma \ref{cube3} yields a contradiction.

Hence $\Pc=\conv(\Dc(K_3,\mathfrak{c}))$ is not nearly Gorenstein, as desired.

\medskip

\noindent
{\bf (Second Step)}
We now turn to the discussion of (3).
Let $\min\{c_1, c_2, c_3\}\geq 2$.  Then $(2c_1-1, 2c_2-1, 2c_3-1)\in \ZZ_{>0}^3$ belongs to 
$2(\conv(\Dc(K_3,\mathfrak{c})) \setminus \partial \conv(\Dc(K_3,\mathfrak{c})))$.
However, it cannot be written as the sum of a point belonging to $\conv(\Dc(K_3,\mathfrak{c})) \cap \ZZ^n$ and a point belonging to $(\conv(\Dc(K_3,\mathfrak{c})) \setminus \partial \conv(\Dc(K_3,\mathfrak{c})))\cap \ZZ^n$.  So, $\conv(\Dc(K_3,\mathfrak{c}))$ is not level.

Let $\min\{c_1, c_2, c_3\}=1$, say,  $c_1=1$. Then $$(\conv(\Dc(K_3,\mathfrak{c})) \setminus \partial \conv(\Dc(K_3,\mathfrak{c})))\cap \ZZ^n=\emptyset.$$
Let $\ab=(a_1,a_2,a_3) \in 
N(\conv(\Dc(K_3,\mathfrak{c})) \setminus \partial \conv(\Dc(K_3,\mathfrak{c})))\cap \ZZ^n$, where $3 \leq N \in \ZZ$.  For each $i=1,2,3$, set $a_i':=\min\{a_i,2c_i-1\}$ and $a_i'':=a_i-a_i'$.  Furthermore, set $\ab':=(a_1',a_2',a_3')$ and $\ab'':=(a_1'',a_2'',a_3'')$. Then $\ab=\ab'+\ab''$. It is clear that $\ab' \in 2(\conv(\Dc(K_3,\mathfrak{c})) \setminus \partial \conv(\Dc(K_3,\mathfrak{c})))$.  To show that $\conv(\Dc(K_3,\mathfrak{c}))$ is level, we must prove that $\ab''\in (N-2)\conv(\Dc(K_3,\mathfrak{c}))$. One has $a_i\leq Nc_i-1$, for each $i\in [3]$. As a consequence, one can easily verify that $0\leq a_i''\leq (N-2)c_i$, for each $i\in [3]$. Hence, to show that $\ab''\in (N-2)\conv(\Dc(K_3,\mathfrak{c}))$, we only need to prove that$$a_1''+a_2''+a_3''\leq (N-2)(c_1+c_2+c_3-1).$$It follows from $c_1=1$ that $a_1'=1$. Thus, $a_1''=a_1-1$. If $a_2\geq 2c_2-1$ and $a_3\geq 2c_3-1$, then
\begin{eqnarray*}
&& a_1''+a_2''+a_3''\\&=  &(a_1-1)+(a_2-2c_2+1)+(a_3-2c_3+1)\\  &= &(a_1+a_2+a_3)-2(c_2+c_3)+1\\
&<&   N(c_1+c_2+c_3-1)-2(c_2+c_3)+1\\ &= &(N-2)(c_2+c_3)+1\\ &=&(N-2)(c_1+c_2+c_3-1)+1,
\end{eqnarray*}
where the third and the fourth equality follow from $c_1=1$. Therefore, suppose that at least one of the inequalities $a_2\geq 2c_2-1$ and $a_3\geq 2c_2$ do not hold. Let $a_3\leq 2c_3-2$. Then $a_3'=a_3$. Consequently, $a_3''=0$. Since $a_i''\leq (N-2)c_i$ for $i=1,2$, we deduce that
\begin{align*}
& a_1''+a_2''+a_3''\leq (N-2)(c_1+c_2)\leq (N-2)(c_1+c_2+c_3-1),
\end{align*}
as desired.
\hspace{11.6cm}
\end{proof}

Now, combining \cite[Theorem 4.5]{HSF_5} and Lemmas \ref{cube}, \ref{cube2}, \ref{cube3} and \ref{n=3} immediately yields the following Theorem \ref{main_A} on $n \geq 4$ and Theorem \ref{main_B} on $n=3$.

\begin{Theorem}
    \label{main_A}
    Let $n \geq 4$.

    \begin{itemize}
        \item[(1)] 
        A lattice polytope $\Pc \subset \RR^n$ of dimension $n$ is a Gorenstein $K_n$-polytope if and only if $\Pc=\Qc_{k_1,\ldots, k_n}$, where either (i) $k_1 = \cdots = k_n = 2$, or (ii) $n$ is even and $k_1 = \cdots = k_n = 1$.
        \item[(2)] 
        A lattice polytope $\Pc \subset \RR^n$ of dimension $n$ is a nearly Gorenstein $K_n$-polytope if and only if $\Pc=\Qc_{k_1,\ldots, k_n}$, where $k_1+\cdots +k_n$ is even, $k_i\leq \sum_{1 \leq j \leq n, \, j \neq i} k_j$ for each $1 \leq i \leq n$, and either each $k_i \leq 2$ or each $k_i \geq 2$.
       
        \item[(3)]
        A lattice polytope $\Pc \subset \RR^n$ of dimension $n$ is a level $K_n$-polytope if and only if $\Pc=\Qc_{k_1,\ldots, k_n}$, where $k_1+\cdots +k_n$ is even and $k_i\leq \sum_{1 \leq j \leq n, \, j \neq i} k_j$ for each $1 \leq i \leq n$.
    \end{itemize}
\end{Theorem}

Let $(c_1,c_2,c_3) \in (\ZZ_{>0})^3$ and $\Sc_{c_1,c_2,c_3} \subset \RR^3$ the lattice polytope of dimension $3$ which is defined by the system of linear inequalities $0\leq x_1\leq c_1, 0\leq x_2\leq c_2, 0\leq x_3\leq c_3$ together with $x_1+x_2+x_3\leq c_1+c_2+c_3-1$.

\begin{Theorem}
    \label{main_B}
    Let $K_3$ be the triangle on $3$ vertices. 

    \begin{itemize}
        \item[(1)] 
        A Gorenstein $K_3$-polytope is $\Qc_{2,2,2}$ and $\Sc_{1,1,1}$. 
        \item[(2)] 
         A nearly Gorenstein $K_3$-polytope is $\Qc_{k_1,k_2,k_3}$, where (i) $k_1+k_2+k_3$ is even, (ii) $k_i \leq (k_1+k_2+k_3)-k_i$ for each $1 \leq i \leq 3$ and (iii) either each $k_i \leq 2$ or each $k_i \geq 2$, and $\Sc_{1,1,1}$ 
        \item[(3)]
         A level $K_3$-polytope is $\Qc_{k_1,k_2,k_3}$, where (i) $k_1+k_2+k_n$ is even and (ii) $k_i \leq (k_1+k_2+k_3)-k_i$ for each $1 \leq i \leq 3$, and 
         $\Sc_{c_1,c_2,c_3}$ with $\min\{c_1, c_2, c_3\}=1$.       
    \end{itemize}
\end{Theorem}

\begin{Corollary}
    \label{NearLevel}
    Let $n \geq 3$.  Every nearly Gorenstein $K_n$-polytope is level.
\end{Corollary}

\section{Complete bipartite graphs}
Let $K_{m,n}$ denote the complete bipartite graph on the vertex set
\[
V(K_{m,n}) = \{x_1, \ldots, x_m\} \cup \{x_{m+1}, \ldots, x_{m+n}\}
\]
and let $\cc=(c_1, \ldots, c_m, c_{m+1}, \cdots, c_{m+n} \in (\ZZ_{>0})^{m+n}$.

As is discussed in \cite[Lemma 2.2]{HSF_6}, if
$$c_1+\cdots +c_m=c_{m+1}+\cdots =c_{m+n},$$ then $\Bc(K_{m,n},\mathfrak{c})=\{x_1^{c_1}\cdots x_{m+n}^{c_{m+n}}\}$ and $\conv(K_{m,n},\mathfrak{c}))=\Qc_{c_1, \ldots, c_{m+n}}$ is a cube.  

Let $c_1+\cdots +c_m\neq c_{m+1}+\cdots =c_{m+n}$, say, $c_1+\cdots +c_m>c_{m+1}+\cdots +c_{m+n}$.  If there is $i\in [m]$ with $c_i> c_{m+1}+\cdots + c_{m+n}$, then one replaces $c_i$ with $c_{m+1}+\cdots + c_{m+n}$. This substitution does not change $\Bc(K_{m,n},\mathfrak{c})$. So, throughout the present section, we suppose that $\cc=(c_1, \ldots, c_m, c_{m+1}, \ldots, c_{m+n}) \in (\ZZ_{>0})^{m+n}$ satisfies 
\begin{itemize}
    \item $c_1+\cdots +c_m>c_{m+1}+\cdots +c_{m+n}$,
    \item
    $c_i\leq c_{m+1}+\cdots +c_{m+n}$ for each $i\in [m]$.
\end{itemize}

We first recall \cite[Lemma 2.2]{HSF_6}.

\begin{Lemma}
\label{Osaka}
Under the above hypothesis, set$$A:=\{i\in [m] : c_i=c_{m+1}+\cdots +c_{m+n}\}, \quad B:=[m]\setminus A.$$ Then $\conv(\Dc(K_{m,n},\mathfrak{c})) \subset \RR^{m+n}$ is defined by the system of linear inequalities 

$\bullet$ $0\leq x_i\leq c_i, \quad i=m+1, \cdots, m+n$;

$\bullet$ $0\leq x_i\leq c_i, \quad i\in B$;

$\bullet$ $0\leq x_i, \quad i\in A$;

$\bullet$ $x_1+\cdots +x_m\leq c_{m+1}+\cdots +c_{m+n}$.
\end{Lemma}

\begin{Lemma} \label{compbips3}
Set$$s:=\max\bigg\{\Big\lfloor\frac{m}{c_{m+1}+\cdots +c_{m+n}}\Big\rfloor+1
\bigg\}.$$If $s\geq 3$, then $\conv(\Dc(K_{m,n},\mathfrak{c}))$ is nearly Gorenstein if and only if the following conditions hold.

\begin{itemize}
    \item [(i)] $s=3$;
    \item [(ii)] $c_{m+1}=\cdots =c_{m+n}=1$;
    \item [(iii)] $c_i\in \{1,n\}$ for each $i\in [m]$ and $|\{i\in [m] : c_i=n\}|\geq 2m-3n+1$. 
\end{itemize}  
\end{Lemma}

\begin{proof}
Let $\Pc=\conv(\Dc(K_{m,n},\mathfrak{c}))$ and recall Lemma \ref{Osaka}.  It follows from the definition of $s$ together with $s\geq 3$ that$$s=\min\{N\geq 1 : N(\Pc\setminus\partial\Pc)\cap \ZZ^{m+n}\neq \emptyset\}.$$
Furthermore, we deduce from the definition of $s$ that$$(s-1)(c_{m+1}+\cdots +c_{m+n})\leq m\leq s(c_{m+1}+\cdots +c_n)-1.$$

\medskip

\noindent
{\bf (``If'')}
Supposing that (i), (ii) and (iii) hold, we show that $\Pc$ is nearly Gorenstein. Let $n=1$. One has $m=2$ and $c_1=c_2=c_3=1$. One easily sees that $\Pc$ is nearly Gorenstein. So, assume that $n\geq 2$. Let $\ab=(a_1, \ldots, a_{m+n})\in\Pc\cap\ZZ^{m+n}$. Using induction on $n-(a_1+\cdots +a_m)$, we show that $\xb^\ab t\in \omega_{K[\Pc]}.\omega_{K[\Pc]}^{-1}$.

Let $a_1+\cdots +a_m=n$. Set $$X:=\{i\in [m] : a_i=c_i=1\}.$$ Then, by (iii), one has $$|X|\leq |\{i\in m: c_i=1\}|\leq m-(2m-3n+1)=3n-m-1.$$This yields that

\begin{eqnarray*}
&&\sum_{i\in [m]\setminus X}a_i-(3n-1)=\sum_{i=1}^ma_i-|X|-(3n-1)\\ &\leq & \sum_{i=1}^ma_i-2|X|-m=\sum_{i=1}^m(a_i-1)-2|X|\\  &=& \sum_{i\in [m]\setminus X}(a_i-1)-2|X|.   
\end{eqnarray*}
Using the above inequality together with $a_1+\cdots +a_m=n$ and $s=3$, one has $$|X|-m\leq \sum_{i\in [m]\setminus X}a_i-(3n-1)+2|X|\leq \sum_{i\in [m]\setminus X}^m(a_i-1).$$These inequalities imply that for each $i\in [m]\setminus X$, there is $a_i'\in \ZZ$ for which $-1\leq a_i'\leq a_i-1$ and$$\sum_{i\in [m]\setminus X}a_i'=\sum_{i\in [m]\setminus X}a_i-(3n-1)+2|X|.$$For each $i\in X\cup ([m+n]\setminus [m])$, set $a_i':=-1$. Define $\ab':=(a_1', \ldots, a_{m+n}')$. Notice that $$\xb^\ab t=(\xb^{\ab-\ab'}t^3)(\xb^{\ab'}t^{-2}).$$By the choice of $\ab'$, one has $\xb^{\ab'}t^{-2}\in \omega_{K[\Pc]}^{-1}$. We show that $\xb^{\ab-\ab'}t^3\in \omega_{K[\Pc]}$. Obviously, each component of $\ab-\ab'$ is positive. If $c_i=1$ for some $i\in [m+n]$, then $a_i-a_i'=a_i+1\leq 2$. Furthermore,
\begin{eqnarray*}
&& (a_1-a_1')+\cdots +(a_m-a_m')\\ &=&(a_1+\cdots +a_m)-(a_1'+\cdots +a_m')\\ & = & n-\Big(\sum_{i\in [m]\setminus X}a_i-(3n-1)+|X|\Big)\\ & = & n-\big(n-(3n-1)\big)=3n-1. 
\end{eqnarray*}
Thus $\xb^{\ab-\ab'}t^3\in \omega_{K[\Pc]}$. This implies that $\xb^\ab t\in \omega_{K[\Pc]}.\omega_{K[\Pc]}^{-1}$, as required.

Let $a_1+\cdots +a_m < n$. Then there is $i\in [m]$ with $x_i\xb^{\ab}t\in K[\Pc]$. Using induction hypothesis, there are $\xb^{\ab_1}t^{\ell_1}\in \omega_{K[\Pc]}$ and $\xb^{\ab_2}t^{\ell_2}\in \omega_{K[\Pc]}^{-1}$ for which$$x_i\xb^{\ab}t=(\xb^{\ab_1}t^{\ell_1})(\xb^{\ab_2}t^{\ell_2}).$$If the $i$th component of $\ab_2$ is nonnegative, then $(\xb^{\ab_2}t^{\ell_2})/x_i\in \omega_{K[\Pc]}^{-1}$. So,$$\xb^{\ab}t=(\xb^{\ab_1}t^{\ell_1})\big((\xb^{\ab_2}t^{\ell_2})/x_i\big)\in \omega_{K[\Pc]}.\omega_{K[\Pc]}^{-1},$$and we are done. Suppose that the $i$th component of $\ab_2$ is negative. It follows from $x_i\xb^{\ab}=\xb^{\ab_1}\xb^{\ab_2}$ that the $i$th component of $\ab_1$ is at least two. Then $(\xb^{\ab_1}t^{\ell_1})/x_i\in \omega_{K[\Pc]}$ and an  argument as above yields that $x_i\xb^{\ab}t\in \omega_{K[\Pc]}.\omega_{K[\Pc]}^{-1}$.

\medskip

\noindent
{\bf (``Only If'')}  Supposing that $\Pc$ is nearly Gorenstein, we show that (i), (ii) and (iii) hold. Let $\bb=(b_1, \ldots, b_{m+n})\in \Pc\cap \ZZ^{m+n}$ satisfy $$b_1+\cdots +b_m=c_{m+1}+\cdots + c_{m+n}.$$Note that since $$c_1+\cdots +c_m>c_{m+1}+\cdots c_{m+n},$$such a point $\bb$ exists. As $\Pc$ is nearly Gorenstein, one has $\xb^{\bb}t\in \omega_{K[\Pc]}.\omega_{K[\Pc]}^{-1}$. Thus, there are $\xb^{\bb'}t^p\in \omega_{K[\Pc]}$ and $\xb^{\bb''}t^{q}\omega_{K[\Pc]}^{-1}$ for which $$\xb^{\bb}t=(\xb^{\bb'}t^{p})(\xb^{\bb''}t^{q})=\xb^{\bb'+\bb''}t^{p+q}.$$ Hence, $\bb'+\bb''=\bb$ and $p+q=1$. Since $x_1\cdots x_{m+n}t^s\in \omega_{K[\Pc]}$ and $\xb^{\bb''}t^{q}\in \omega_{K[\Pc]}^{-1}$, one has
\begin{eqnarray}
\label{eeeeeeee}   
(x_1\cdots x_{m+n}t^s)(\xb^{\bb''}t^{q})\in K[\Pc].
\end{eqnarray}
In particular, $q\geq -s$. Let $q=-s$. Then, by (\ref{eeeeeeee}), one has  $\xb^{\bb''}=x_1^{-1}\cdots x_{m+n}^{-1}$. As $s\geq 3$, one has $x_1\cdots x_{m+n-1}x_{m+n}^2t^s\in \omega_{K[\Pc]}$. Since $\xb^{\bb''}t^{q}\in\omega_{K[\Pc]}^{-1}$, we obtain that 
\begin{eqnarray*}
x_{m+n}&=&(x_1x_2\cdots x_{m+n-1}x_{m+n}^2t^s)(x_1^{-1}\cdots x_{m+n}^{-1}t^{-s})\\ & =&(x_1x_2\cdots x_{m+n-1}x_{m+n}^2t^s)(\xb^{\bb''}t^{q})\in K[\Pc],
\end{eqnarray*}
a contradiction. Hence $q\geq -s+1$.

Set $\bb':=(b_1', \ldots, b_{m+n}')$ and $\bb'':=(b_1'', \ldots, b_{m+n}'')$. Since $p\geq s$, the equality $p+q=1$ implies that $(p,q)=(s,-s+1)$. By (\ref{eeeeeeee}), one has $b_i''\geq -1$ for each $i\in [m+n]$. Note that for any $j\in [m+n]\setminus [m]$, we have $x_j^{sc_j-2}(x_1\cdots x_{m+n})t^s\in\omega_{K[\Pc]}$. Since $\xb^{\bb''}t^q=\xb^{\bb''}t^{-s+1}\in\omega_{K[\Pc]}^{-1}$, one has$$\big(x_j^{sc_j-2}(x_1\cdots x_{m+n})t^s\big)(\xb^{\bb''}t^{-s+1})\in K[\Pc].$$As $b_i''\geq -1$, the exponent of $x_j$ in the above product is at least $sc_j-2$. It follows that $sc_j-2\leq c_j$. This implies that $s=3$ and $c_j=1$ for each $j\in [m+n]\setminus [m]$. Therefore, $(p,q)=(3,-2)$.

Let $2\leq c_i< c_{m+1}+\cdots +c_{m+n}=n$ for some $i\in [m]$. Then$$(Nc_i-1)+(m-1)< Nn,$$for some $N\geq s$. As a consequence, $x_i^{Nc_i-2}(x_1\cdots x_{m+n})t^N\in \omega_{K[\Pc]}$. Since $\xb^{\bb''}t^q=\xb^{\bb''}t^{-2}\in\omega_{K[\Pc]}^{-1}$, one has $$\big(x_i^{Nc_i-2}(x_1\cdots x_{m+n})t^N\big)(\xb^{\bb''}t^{-2})\in K[\Pc].$$However, as $b_i''\geq -1$, the exponent of $x_i$ in the above product is at least$$Nc_i-2> (N-2)c_i,$$where the inequality follows from $c_i\geq 2$. This is a contradiction. Hence $c_i\in\{1,n\}$ for each $i\in [m]$. 

If $n=1$, then $m=2$ and $c_i=1$ for each $i\in [m+n]$. So, (i), (ii) and (iii) hold. Let $n\geq 2$. Assume that there is $i\in [m]$ with $c_i=1$. Note that $x_i^2(x_1\cdots x_{m+n})t^4\in\omega_{K[\Pc]}$. Since $\xb^{\bb''}t^q=\xb^{\bb''}t^{-2}\in\omega_{K[\Pc]}^{-1}$, one has$$\big(x_i^2(x_1\cdots x_{m+n})t^4\big)(\xb^{\bb''}t^{-2})\in K[\Pc].$$As the exponent of $x_i$ in the above product is at most $2$, we conclude that $b_i''\leq -1$. On the other hand, (\ref{eeeeeeee}) yields that $b_i''\geq -1$. Thus, $b_i''=-1$.

Everything we proved so far is valid for each $\bb=(b_1, \ldots, b_{m+n})\in \Pc\cap \ZZ^{m+n}$ satisfying$$b_1+\cdots +b_m=c_{m+1}+\cdots +c_{m+n}=n.$$In what follows, we choose an appropriate $\bb\in \Pc\cap \ZZ^{m+n}$ satisfying the above equality. More precisely, suppose that $$|\{i\in [m] : c_i=n\}|< 2m-3n+1.$$We deduce from the above inequality that $$|\{i\in [m] : c_i=1\}|\geq 3n-m.$$ Recall that $3n-m\leq n$. Without loss of generality, assume that  $c_1=\cdots=c_{3n-m}=1$. Let $\bb\in \Pc\cap \ZZ^{m+n}$ satisfy $b_1=\cdots=b_{n}=b_{m+1}=\cdots=b_{m+n}=1$ and the other components of $\bb$ are zero.  Since $b_i''=-1$ for each $i\in [3n-m]$, we deduce from $\bb=\bb'+\bb''$ that $b_1'=\cdots =b_{3n-m}'=2$. Therefore,
\begin{eqnarray*}
b_1'+\cdots +b_m'&=& \, (b_1'+\cdots +b_{3n-m}')+(b_{3n-m+1}'+\cdots +b_m')\\ &\geq &  2(3n-m)+\big(m-(3n-m)\big)=3n,    
\end{eqnarray*}
contradicting $\xb^{\bb}t^3\in \omega_{K[\Pc]}$. Thus $|\{i\in [m] : c_i=n\}|\geq 2m-3n+1$.
\, \, \, \, \, \,
\end{proof}

\begin{Lemma} \label{compbiplevs3}
Under the situation of Lemma \ref{compbips3}, if $\conv(\Dc(K_{m,n},\mathfrak{c}))$ is nearly Gorenstein, then it is level.   
\end{Lemma}

\begin{proof}
We follow Lemma \ref{compbips3} and its proof. Since $\Pc$ is nearly Gorenstein, we deduce that Lemma \ref{compbips3} (i), (ii) and (iii) hold. Let $n=1$. Then $m=2$ and $c_1=c_2=c_3=1$. So, $\Pc$ is level. Let $n\geq 2$. We may assume that there is an integer $0\leq k\leq m$ with $c_1=\cdots =c_k=n$ and $c_{k+1}=\cdots =c_{m+n}=1$.  Let $A$ and $B$ be the sets defined in Lemma \ref{Osaka}. We know that $\Pc$ is defined by the system of linear inequalities 

$\bullet$ $0\leq x_i\leq 1, \quad i=k+1, \cdots, m+n$;

$\bullet$ $0\leq x_i, \quad  i\in [k]$;

$\bullet$ $x_1+\cdots +x_m\leq n$.   

\noindent
Let $\ab=(a_1, \ldots, a_{m+n})\in N(\Pc\setminus\partial\Pc)\cap\ZZ^{m+n}$ for some integer $N\geq 4$. We must show that there are$$\ab'=(a_1', \ldots, a_{m+n}')\in3(\Pc\setminus\partial\Pc)\cap\ZZ^{m+n}$$ and$$\ab''=(a_1'', \ldots, a_{m+n}'')\in (N-3)\Pc \cap \ZZ^{m+n}$$with $\ab=\ab'+\ab''$.

We know from Lemma \ref{compbips3} (iii) that $k\geq 2m-3n+1$. Therefore, $$k+2(m-k)\leq 3n-1.$$For each $i\in [m+n]\setminus [k]$, set $a_i':=\min\{2,a_i\}$. We deduce from the above inequality that for each $i\in [k]$, there is an integer $a_i'$ with $1\leq a_i'\leq a_i$ for which at least one of the following conditions holds. 

\begin{itemize}
\item [(1)] $a_1'+\cdots +a_m'=3n-1$.
\item [(2)] $a_1'+\cdots +a_m'< 3n-1$ and $a_i'=a_i$, for each $i\in [k]$.
\end{itemize}
In both cases $\ab'=(a_1', \ldots, a_{m+n}')\in 3(\Pc\setminus\partial\Pc)\cap\ZZ^{m+n}$. 

For each $i\in [m+n]$, set $a_i'':=a_i-a_i'$ and put $\ab'':=(a_1'', \ldots, a_{m+n}'')$. To complete our proof, we show that $\ab'':=\ab-\ab'\in (N-3)\Pc$.  It follows that

$\bullet$ $a_i''\geq 0$, for each $i\in [m+n]$, and 

$\bullet$ $a_i''\leq N-3$, for each $i=[m+n]\setminus [k]$. 

\noindent
So, we only need to show that$$a_1''+\cdots +a_m''\leq (N-3)n.$$
If (1) holds, then
\begin{eqnarray*}
a_1''+\cdots +a_m''& = &(a_1+\cdots +a_m)-(a_1'+\cdots +a_m')\\ & \leq & Nn-1-(3n-1)=(N-3)n.   
\end{eqnarray*}
If (2) holds, then set$$X:=\{i\in [m]\setminus [k]: a_i> 2\}.$$For each $i\in [m]\setminus [k]$, one has $a_i''>0$ if and only if $i\in X$. Furthermore,$$|X|\leq m-k\leq m-(2m-3n+1)=3n-1-m\leq n-1,$$where the last inequality follows from $s=3$. Since (2) holds, we have $a_i'=a_i$, for each $i\in [k]$. Therefore, $a_1''=\cdots =a_k''=0$. Hence,
\begin{eqnarray*}
&& a_1''+\cdots +a_m''=\sum_{i\in X}a_i''=\sum_{i\in X}(a_i-a_i')=\sum_{i\in X}(a_i-2)\\ & \leq & \sum_{i\in X}(N-3)=(N-3)|X|\leq (N-3)(n-1)\leq (N-3)n.
\end{eqnarray*}
This completes the proof. 
\hspace{9.1cm}
\end{proof}

\begin{Lemma} \label{compbips2} 
Set$$s:=\max\bigg\{\Big\lfloor\,\frac{m}{c_{m+1}+\cdots +c_{m+n}}\,\Big\rfloor+1, \,  \max\Big\{\Big\lfloor\frac{5-c_i}{2}\Big\rfloor: i\in [m+n]\Big\}\bigg\}.$$If $s=2$, then $\conv(\Dc(K_{m,n},\mathfrak{c}))$ is nearly Gorenstein if and only if one of the following conditions holds.

\begin{itemize}
    \item [($\alpha$)] {\em (i)} $m=2n-1$, {\em (ii)} $c_{m+1}=\cdots =c_{m+n}=1$ and {\em (iii)} $c_i\in \{1,n\}$ for each $i\in [m]$;
    \item [($\beta$)] {\em (i)} $m=2(c_{m+1}+\cdots +c_{m+n})-1$, {\em (ii)} $c_i\in\{1, c_{m+1}+\cdots +c_{m+n}\}$ for each $i\in [m]$ and {\em (iii)} $c_j\in\{1,2\}$ for each $j\in [m+n]\setminus [m]$;
    \item [($\gamma$)] {\em (i)} $m=(c_{m+1}+\cdots +c_{m+n})-1$, {\em (ii)} $c_i\in\{2, c_{m+1}+\cdots +c_{m+n}\}$ for each $i\in [m]$ and {\em (iii)} $c_j\in\{1,2\}$ for each $j\in [m+n]\setminus [m]$;
    \item [($\delta$)] {\em (i)} $m\geq c_{m+1}+\cdots +c_{m+n}+\big |\{i\in [m] : c_i=1\}\big |-1$, {\em (ii)} $c_{j} \in \{1,2\}$ for each $j\in [m+n]\setminus [m]$, {\em (iii)} $c_i\in\{1,2, c_{m+1}+\cdots +c_{m+n}\}$ for each $i\in [m]$ and {\em (iv)} the inequality
    \begin{eqnarray*}
    &&2\big |\{i\in [m] :c_i=1 \ {\rm or} \ c_i = c_{m+1}+\cdots +c_{m+n}\}\big |\\
    & \geq &3m-2(c_{m+1}+\cdots +c_{m+n})+1  
    \end{eqnarray*}   
    holds.
\end{itemize}  
\end{Lemma}

\begin{proof}
Set $\Pc:=\conv(\Dc(G,\mathfrak{c}))$ and and recall Lemma \ref{Osaka}.  It follows from the definition of $s$ together with $s=2$ that$$2=\min\{N\geq 1 : N(\Pc\setminus\partial\Pc)\cap \ZZ^{m+n}\neq \emptyset\}.$$

\medskip

\noindent
{\bf (``If'')}
If ($\alpha$) holds, then \cite[Theorem 5.3]{HSF_5} says that $\Pc$ is Gorenstein. Supposing that one of the conditions ($\beta$), ($\gamma$) and ($\delta$) hold, we show that $\Pc$ is nearly Gorenstein. 
Let $\ab=(a_1, \ldots, a_{m+n})$ belong to $\Pc\cap\ZZ^{m+n}$ and set $k_i:=\min\{a_i+1, 2c_i-1\}$ for each $i\in [m+n]\setminus [m]$.  If ($\beta$) holds, then $$\xb^{\ab}t=\big(x_1\cdots x_m x_{m+1}^{k_{m+1}}\cdots x_{m+n}^{k_{m+n}}t^2\big)\bigg(\frac{\xb^{\ab}}{x_1\cdots x_m x_{m+1}^{k_{m+1}}\cdots x_{m+n}^{k_{m+n}}t}\bigg)$$belongs to $\omega_{K[\Pc]}.\omega_{K[\Pc]}^{-1}$.
If ($\gamma$) holds, then $$\xb^{\ab}t=\big(x_1^{a_1+1}\cdots x_m^{a_m+1} x_{m+1}^{k_{m+1}}\cdots x_{m+n}^{k_{m+n}}t^2\big)\bigg(\frac{1}{x_1\cdots x_m x_{m+1}^{k_{m+1}-a_{m+1}}\cdots x_{m+n}^{k_{m+n}-a_{m+n}}t}\bigg)$$belongs to $\omega_{K[\Pc]}.\omega_{K[\Pc]}^{-1}$.

Let ($\delta$) hold. Using induction on$$(c_{m+1}+\cdots c_{m+1})-(a_1+\cdots +a_m),$$we show that $\xb^\ab t\in \omega_{K[\Pc]}.\omega_{K[\Pc]}^{-1}$.

Let $a_1+\cdots +a_m=c_{m+1}+\cdots +c_{m+n}$. If $c_{m+1}+\cdots + c_{m+n}=1$, then $n=1$ and $c_{m+1}=1$. Since $c_i\leq c_{m+1}+\cdots +c_{m+n}$ for each $i\in [m]$, we deduce that $c_1=\cdots =c_m=1$. Furthermore, it follows from $c_1+\cdots +c_m > c_{m+1}+\cdots +c_{m+n}$ that $m\geq 2$. This contradicts $s=2$. So, $c_{m+1}+\cdots + c_{m+n}\geq 2$. If $c_{m+1}+\cdots + c_{m+n}=2$, then we conclude from $s=2$ that $m\leq 3$. As $c_i\leq c_{m+1}+\cdots +c_{m+n}$ for each $i\in [m]$, we deduce that $c_1, \ldots, c_m$ are at most two. In these cases, one can easily prove that $\Pc$ is nearly Gorenstein.  Let $c_{m+1}+\cdots +c_{m+n}\geq 3$. Set $$X_1:=\{i\in [m] : a_i=c_i=1\}, \quad X_2:=\{i\in [m] : a_i=c_i=2\},$$and$$X_3=\{i\in [m] : a_i=1, c_i=2\}.$$Then, by (iv), one has 
\begin{align*}
2|X_2|+2|X_3| \, \leq &\,\, 2|\{i\in m: c_i=2\}|\leq2 m-\big(3m-2(c_{m+1}+\cdots +c_{m+n})+1\big)\\ =& \,\,2(c_{m+1}+\cdots +c_{m+n})-m-1.    
\end{align*}
This yields that

\begin{eqnarray*}
 &&\sum_{i\in [m]\setminus (X_1\sqcup X_2\sqcup X_3)}a_i-2(c_{m+1}+\cdots +c_{m+n})+1\\ &=&\sum_{i=1}^ma_i-|X_1|-2|X_2|-|X_3|-2(c_{m+1}+\cdots +c_{m+n})+1\\ &\leq& \sum_{i=1}^ma_i-|X_1|-4|X_2|-3|X_3|-m=\sum_{i=1}^m(a_i-1)-|X_1|-4|X_2|-3|X_3|\\  &=& \sum_{i\in [m]\setminus (X_1\sqcup X_2\sqcup X_3)}(a_i-1)-|X_1|-3|X_2|-3|X_3|\\ &\leq& \sum_{i\in [m]\setminus (X_1\sqcup X_2\sqcup X_3)}(a_i-1)-|X_1|-3|X_2|-2|X_3|.   
\end{eqnarray*}
Using the above inequality and (i) together with $a_1+\cdots +a_m=c_{m+1}+\cdots +c_{m+n}$, one has
\begin{eqnarray*}
&& |X_1\sqcup X_2\sqcup X_3|-m
\\ &\leq &\sum_{i\in [m]\setminus (X_1\sqcup X_2\sqcup X_3)}a_i-2(c_{m+1}+\cdots +c_{m+n})+1+|X_1|+3|X_2|+2|X_3|\\ & \leq &\sum_{i\in [m]\setminus (X_1\sqcup X_2\sqcup X_3)}^m(a_i-1).   
\end{eqnarray*}
These inequalities imply that for each $i\in [m]\setminus (X_1\sqcup X_2\sqcup X_3)$, there is an integer $-1\leq a_i'\leq a_i-1$ for which$$\sum_{i\in [m]\setminus (X_1\sqcup X_2\sqcup X_3)}a_i'=\sum_{i\in [m]\setminus (X_1\sqcup X_2\sqcup X_3)}a_i-2(c_{m+1}+\cdots +c_{m+n})+1+|X_1|+3|X_2|+2|X_3|.$$For each $i\in X_2\cup X_3$, set $a_i':=-1$. Put $a_i':=0$, for each $i\in X_1$. Also, set $a_i'=\min\{1-c_i,a_i-1\}$, for each $i\in [m+n]\setminus [m]$.  Define $\ab':=(a_1', \ldots, a_{m+n}')$. Notice that $$\xb^\ab t=(\xb^{\ab-\ab'}t^2)(\xb^{\ab'}t^{-1}).$$By the choice of $\ab'$, one has $\xb^{\ab'}t^{-1}\in \omega_{K[\Pc]}^{-1}$. We show that $\xb^{\ab-\ab'}t^2\in \omega_{K[\Pc]}$. Clearly, each component of $\ab-\ab'$ is positive. If $c_i=1$ for some $i\in [m+n]$, then $0\leq a_i\leq 1$, and so, $a_i-a_i'\leq 1$. Similarly, if $c_i=2$ for some $i\in [m+n]$, then $a_i-a_i'\leq 3$. Furthermore,
\begin{eqnarray*}
& &(a_1-a_1')+\cdots +(a_m-a_m')\\&=&(a_1+\cdots +a_m)-(a_1'+\cdots +a_m')\\ & =&(c_{m+1}+\cdots + c_{m+n})\\ &=&-\Big(\sum_{i\in [m]\setminus (X_1\sqcup X_2\sqcup X_3)}a_i-2(c_{m+1}+\cdots +c_{m+n})+1+|X_1|+2|X_2|+|X_3|\Big)\\ &=&2(c_{m+1}+\cdots +c_{m+n})-1,  
\end{eqnarray*}
where the last inequality follows from $a_1+\cdots +a_m=c_{m+1}+\cdots +c_{m+n}$. Thus, $\xb^{\ab-\ab'}t^2\in \omega_{K[\Pc]}$. This implies that $\xb^\ab t\in \omega_{K[\Pc]}.\omega_{K[\Pc]}^{-1}$, as required.

Now, assume that $a_1+\cdots +a_m < c_{m+1}+\cdots +c_{m+n}$. Then there is $i\in [m]$ with $x_i\xb^{\ab}t\in K[\Pc]$. Using induction hypothesis, it follows that there are $\xb^{\ab_1}t^{\ell_1}\in \omega_{K[\Pc]}$ and $\xb^{\ab_2}t^{\ell_2}\in \omega_{K[\Pc]}^{-1}$ for which$$x_i\xb^{\ab}t=(\xb^{\ab_1}t^{\ell_1})(\xb^{\ab_2}t^{\ell_2}).$$If the $i$th component of $\ab_2$ is nonnegative, then $\xb^{\ab_2}t^{\ell_2}/x_i\in \omega_{K[\Pc]}^{-1}$. In this case$$\xb^{\ab}t=(\xb^{\ab_1}t^{\ell_1})(\xb^{\ab_2}t^{\ell_2}/x_i)\in \omega_{K[\Pc]}.\omega_{K[\Pc]}^{-1},$$and we are done. Let the $i$th component of $\ab_2$ be negative. It then follows from $x_i\xb^{\ab}=\xb^{\ab_1}\xb^{\ab_2}$ that the $i$th component of $\ab_1$ is at least two. Then $(\xb^{\ab_1}t^{\ell_1})/x_i\in \omega_{K[\Pc]}$ and a similar argument as above yields that $\xb^{\ab}t\in \omega_{K[\Pc]}.\omega_{K[\Pc]}^{-1}$.

\medskip

\noindent
{\bf (``Only If'')}
Supposing that $\Pc$ is nearly Gorenstein, we show that one of the conditions ($\alpha$), ($\beta$), ($\gamma$) and ($\delta$) holds. If $|2(P\setminus \partial\Pc)\cap\ZZ^{m+n}|=1$. then $\Pc$ is pseudo-Gorenstein. In this case, we deduce from \cite{Miyashita} that $\Pc$ is Gorenstein. Hence, it follows from \cite[Theorem 5.3]{HSF_5} that ($\alpha$) holds.  Let $|2(\Pc\setminus \partial\Pc)\cap\ZZ^{m+n}|\geq 2$.

Let $\bb=(b_1, \ldots, b_{m+n})\in \Pc\cap \ZZ^{m+n}$ satisfy$$b_1+\cdots +b_m=c_{m+1}+\cdots + c_{m+n}.$$Since $$c_1+\cdots +c_m>c_{m+1}+\cdots c_{m+n},$$such a point $\bb$ exists. As $\Pc$ is nearly Gorenstein, one has $\xb^{\bb}t\in \omega_{K[\Pc]}.\omega_{K[\Pc]}^{-1}$. So, there are $\xb^{\bb'}t^p\in \omega_{K[\Pc]}$ and $\xb^{\bb''}t^{q}\in\omega_{K[\Pc]}^{-1}$ for which$$\xb^{\bb}t=(\xb^{\bb'}t^{p})(\xb^{\bb''}t^{q})=\xb^{\bb'+\bb''}t^{p+q}.$$Therefore, $\bb'+\bb''=\bb$ and $p+q=1$. Since $x_1\cdots x_{m+n}t^2\in \omega_{K[\Pc]}$ and $\xb^{\bb''}t^{q}\in \omega_{K[\Pc]}^{-1}$, one has
\begin{eqnarray}
\label{eeeeeeeee}   
(x_1\cdots x_{m+n}t^2)(\xb^{\bb''}t^{q})\in K[\Pc].
\end{eqnarray}
In particular, $q\geq -2$. Let $q=-2$. It follows from (\ref{eeeeeeeee}) that $\xb^{\bb''}=x_1^{-1}\cdots x_{m+n}^{-1}$. We deduce from $|2(\Pc\setminus\partial\Pc)|\geq 2$ that $x_i(x_1, \cdots x_{m+n})t^2\in \omega_{K[\Pc]}$ for some $i\in [m+n]$. It then follows from $\xb^{\bb''}t^{q}\in\omega_{K[\Pc]}^{-1}$ that
\begin{align*}
& x_i=x_i(x_1\cdots x_{m+n}t^2)(x_1^{-1}\cdots x_{m+n}^{-1}t^{-2})=x_i(x_1x_2\cdots x_{m+n}t^2)(\xb^{\bb''}t^{q})\in K[\Pc],
\end{align*}
a contradiction. Thus $q\geq -1$. Since $p\geq 2$ and $p+q=1$, we conclude that $(p,q)=(2,-1)$. The same argument as in the proof of Lemma \ref{compbips3} yields that each of $c_{m+1}, \ldots, c_{m+n}$ belongs to $\{1,2\}$ and$$c_j\in\{1,2, c_{m+1}+\cdots +c_{m+n}\}$$ for each $j\in [m]$.

If $c_{m+1}+\cdots + c_{m+n}=1$, then $n=1$ and $c_{m+1}=1$. Since $c_i\leq c_{m+1}+\cdots +c_{m+n}$ for each $i\in [m]$, we deduce that $c_1=\cdots =c_m=1$. Furthermore, it follows from $c_1+\cdots +c_m > c_{m+1}+\cdots +c_{m+n}$ that $m\geq 2$. This contradicts $s=2$. Thus $c_{m+1}+\cdots + c_{m+n}\geq 2$. If $c_{m+1}+\cdots + c_{m+n}=2$, then we conclude from $s=2$ that $m\leq 3$. So, ($\beta$) or ($\delta$) holds. Let $c_{m+1}+\cdots + c_{m+n}\geq 3$.

Set $\bb':=(b_1', \ldots, b_{m+n}')$ and $\bb'':=(b_1'', \ldots, b_{m+n}'')$. Assume that there is $i\in [m]$ with $c_i=2$. The inequality $c_{m+1}+\cdots + c_{m+n}\geq 3$ implies that $x_i^6(x_1\cdots x_{m+n})t^4$ belongs to $\omega_{K[\Pc]}$. Since $\xb^{\bb''}t^q=\xb^{\bb''}t^{-1}\in\omega_{K[\Pc]}^{-1}$, it follows that$$\big(x_i^6(x_1\cdots x_{m+n})t^4\big)(\xb^{\bb''}t^{-1})\in K[\Pc].$$Since the exponent of $x_i$ in the above product is at most $6$, we conclude that $b_i''\leq -1$. On the other hand, (\ref{eeeeeeeee}) yields that $b_i''\geq -1$. Hence $b_i''=-1$.

As in the proof of Lemma \ref{compbips3}, everything we proved so far is valid for each $\bb=(b_1, \ldots, b_{m+n})\in \Pc\cap \ZZ^{m+n}$ with$$b_1+\cdots +b_m=c_{m+1}+\cdots +c_{m+n}.$$In what follows, we choose an appropriate $\bb\in \Pc\cap \ZZ^n$ satisfying the above equality.

\medskip

\noindent
{\bf Claim.} If $m=2(c_{m+1}+\cdots +c_{m+n})-1$, then $c_i\neq 2$, for each $i\in [m]$.

{\it Proof of the claim.} Assume that there is some $i\in [m]$ with $c_i=2$. Without loss of generality, we may assume that $i=1$. Since $c_{m+1}+\cdots +c_{m+n}\geq 3$, there is $\bb=(b_1, \ldots, b_{m+n})\in \Pc\cap \ZZ^{m+n}$ for which $b_1=2$ and$$b_1+\cdots +b_m=c_{m+1}+\cdots +c_{m+n}.$$As $b_1''=-1$, it follows from $b_1''+b_1'=b_1=2$ that $b_1'=3$. Thus,$$b_1'+\cdots +b_m'\geq 3+(m-1)=m+2=2(c_{m+1}+\cdots +c_{m+n})+1,$$contradicting $\xb^{\bb'}t^2\in\omega_{K[\Pc]}$. This proves the claim.

\medskip

To complete the proof of the lemma, we consider the following cases.

\medskip

\noindent
{\bf (Case 1)} Let $\big | \{i\in [m] : c_i=1\}\big |\geq c_{m+1}+\cdots +c_{m+n}+1$. Set $\ell:=c_{m+1}+\cdots + c_{m+n}$. We may assume that $c_1=\cdots =c_{\ell+1}=1$. Let $\bb=(b_1, \ldots, b_{m+n})\in \Pc\cap \ZZ^{m+n}$ be the vector defined by $b_1=\cdots =b_{\ell}=1$ and $b_i=0$ for each $i\in [m+n]\setminus [\ell]$. Since $b_j'=1$, for each $j\in [\ell]$, we deduce that $b_j''=0$ for each $j\in [\ell]$. As $b_i''\geq -1$, for each $i\in [m+n]\setminus [\ell]$, one has $$b_1''+\cdots +b_m''\geq \ell-m.$$It follows from$$c_1+\cdots +c_m> (c_{m+1}+\cdots + c_{m+n})=\ell$$that there exists an integer $N\geq 2$ with $$(Nc_1-1)+\cdots + (Nc_m-1)\geq N\ell-1.$$Therefore, there is $\db=(d_1, \cdots, d_{m+n})\in N(\Pc\setminus\partial\Pc) \cap \ZZ^{m+n}$ with$$d_1+\cdots + d_m=N\ell-1.$$Since $\xb^{\bb''}t^{-1}\in \omega_{K[\Pc]}^{-1}$, one has $(\xb^{\bb''}t^{-1})(\xb^{\db}t^N)\in K[\Pc]$. However,
\begin{eqnarray*}
& & (b_1''+d_1)+\cdots +(b_m''+d_m)
\\ & = & (b_1''+\cdots +b_m'')+(d_1+\cdots +d_m)
\\  & \geq & \ell-m+N\ell-1\\
& =& (N+1)\ell-m-1.   
\end{eqnarray*}
Consequently, $(N+1)\ell-m-1\leq (N-1)\ell$. This yields that $m\geq 2\ell-1$. Since $s=2$, we deduce that $m=2\ell-1$. Hence the claim shows that ($\beta$) holds. 

\medskip

\noindent
{\bf (Case 2)} Let $\big | \{i\in [m] : c_i=1\}\big |\leq c_{m+1}+\cdots +c_{m+n}$. As in Case 1, set $\ell:=c_{m+1}+\cdots +c_{m+n}$.  Suppose that$$m\leq \ell+\big |\{i\in [m]: c_i=1\}\big |-2.$$Since $c_1 +\cdots +c_m> \ell$ and 
$\big | \{i\in [m] : c_i=1\}\big |\leq \ell$, there is a monomial $u$ on the variables $x_1, \ldots, x_m$ with  $ut\in K[\Pc]$ and $\deg(u)=\ell$ for which $\deg_{x_i}(u)=1$ for each $i\in [m]$ with $c_i=1$. Let $\bb=(b_1, \ldots, b_{m+n})$ denote the exponent vector of $u$. In particular, $\bb\in \Pc \cap\ZZ^{m+n}$ and $b_1+\cdots +b_m=\ell$. For each $i\in [m]$ with $c_i=1$, one has $b_i'=1$. Hence, the equality $\bb'+\bb''=\bb$ implies that $b_i''=0$. Since $b_j''\geq -1$ for each $j\in [m]$, we conclude that
\begin{align*}
& b_1''+\cdots + b_m''\geq -m+\big |\{i\in [m] : c_i=1\}\big |\geq 2-\ell.
\end{align*}
As in Case 1, there exists an integer $N\geq 2$ and $\db\in N(\Pc\setminus \partial\Pc) \cap \ZZ^{m+n}$ with$$d_1+\cdots +d_m=N\ell-1.$$It follows from $\xb^{\bb''}t^{-1}\in\omega_{K[\Pc]}^{-1}$ that$$\xb^{\db+\bb''}t^{N-1}=(\xb^{\db}t^N)(\xb^{\bb''}t^{-1})\in K[\Pc].$$However,
\begin{eqnarray*}
& & (d_1+b_1'')+\cdots + (d_m+b_m'')\\
& = &(d_1+\cdots +d_m)+(b_1''+\cdots +b_m'')\\ & \geq &N\ell-1+2-\ell=(N-1)\ell+1,
\end{eqnarray*}
a contradiction. Thus
\begin{eqnarray}
\label{eeeeeeeeeeeeeee}   
m\geq\ell+\big |\{i\in [m]: c_i=1\}\big |-1.
\end{eqnarray}
If$$2\big |\{i\in [m] :c_i=1 \ {\rm or} \ c_i=c_{m+1}+\cdots +c_{m+n}\}\big |\geq 3m-2\ell+1,$$then ($\delta$) holds. Suppose that$$2\big |\{i\in [m] :c_i=1 \ {\rm or} \ c_i=c_{m+1}+\cdots +c_{m+n}\}\big |\leq 3m-2\ell.$$The above inequality yields that$$2\big |\{i\in [m] :c_i=2\}\big |\geq 2m-(3m-2\ell)=2\ell-m.$$Set$$\ell':=\min\Big\{\Big\lfloor\frac{\ell}{2}\Big\rfloor, \big |\{i\in [m] :c_i=2\}\big |\Big\}.$$Since $\big |\{i\in [m] :c_i=2\}\big |\geq \ell'$, we may assume that $c_1=\cdots =c_{\ell'}=2$.

\smallskip

{\bf Subcase 2.1.} Suppose $\ell'= \big |\{i\in [m] :c_i=2\}\big |$. It follows from $\ell'\leq \lfloor\ell/2\rfloor$ and $c_1+\cdots +c_m >\ell$ that there is $\bb=(b_1, \cdots, b_{m+n})\in \Pc\cap \ZZ^{m+n}$ with $b_1=\cdots=b_{\ell'}=2$ and $b_1+\cdots +b_m=\ell$. As $b_i''=-1$ for each $i\in [\ell']$, we deduce from $\bb=\bb'+\bb''$ that $b_1'=\cdots =b_{\ell'}'=3$. Therefore,
\begin{eqnarray*}
b_1'+\cdots +b_m'& \geq &3\ell'+(m-\ell') = m+2\ell' \\ & =&m+ 2\big |\{i\in [m] :c_i=2\}\big |\\ & \geq &m+(2\ell-m)=2\ell,  
\end{eqnarray*}
contradicting $\xb^{\bb'}t^2\in \omega_{K[\Pc]}$.

\smallskip

{\bf Subcase 2.2.} Suppose that $\ell$ is even and $\ell'=\ell/2$. Let $\bb=(b_1, \cdots, b_{m+n})$ be the vector with $b_1=\cdots =b_{\ell'}=2$ and $b_i=0$ for each $i\in [m+n]\setminus [\ell']$. Then $\bb\in \Pc\cap \ZZ^{m+n}$ and $b_1+\cdots +b_m=\ell$. Since $b_i''=-1$ for each $i\in [\ell']$, we deduce from $\bb=\bb'+\bb''$ that $b_1'=\cdots =b_{\ell'}'=3$. Therefore,
\begin{eqnarray*}
b_1'+\cdots +b_m' &\geq & 3\ell'+(m-\ell')=m+2\ell'=m+\ell\\ & \geq & 2\ell+\big |\{i\in [m]: c_i=1\}\big |-1,  
\end{eqnarray*}
where the last inequality follows from inequality (\ref{eeeeeeeeeeeeeee}). We conclude from $\xb^{\bb'}t^2\in \omega_{K[\Pc]}$ that $b_1'+\cdots +b_m'\leq 2\ell-1$. Therefore, the above inequalities imply that $c_i\neq 1$ for each $i\in [m]$ and, in addition, $m=\ell-1$. In particular, ($\gamma$) holds.

\smallskip

{\bf Subcase 2.3.} Suppose that $\ell$ is odd and $\ell'=(\ell-1)/2$. Seeing Subcase 2.1, we may assume that$$\big |\{i\in [m] :c_i=2\}\big |\geq \ell'+1.$$Let $c_{\ell'+1}=2$ and let $\bb=(b_1, \cdots, b_{m+n})$ be the vector with $b_1=\cdots =b_{\ell'}=2$, $b_{\ell'+1}=1$ and $b_i=0$ for each $i\in [m+n]\setminus [\ell'+1]$. Then $\bb\in \Pc\cap \ZZ^{m+n}$ and $b_1+\cdots +b_m=\ell$. Since $b_i''=-1$ for each $i\in [\ell'+1]$, we deduce from $\bb=\bb'+\bb''$ that $b_1'=\cdots =b_{\ell'}'=3$ and $b_{\ell'}'=2$. Therefore,
\begin{align*}
& b_1'+\cdots +b_m'\geq 3\ell'+2+(m-\ell'-1)=m+2\ell'+1=m+\ell.
\end{align*}
Then the same argument as in Subcase 2.2 shows that ($\gamma$) holds.
\hspace{2cm}
\end{proof}

\begin{Lemma} \label{compbiplevs2}
Under the situation of Lemma \ref{compbips2}, if $\conv(\Dc(K_{m,n},\mathfrak{c}))$ is nearly Gorenstein, then it is level.   
\end{Lemma}

\begin{proof}
We follow Lemma \ref{compbips2} and its proof. Supposing that $\ab=(a_1, \ldots, a_{m+n})$ belongs to $N(\Pc\setminus\partial\Pc)\cap\ZZ^{m+n}$ for some integer $N\geq 3$, we show that there are$$\ab'=(a_1', \ldots, a_{m+n}')\in 2(\Pc\setminus\partial\Pc)\cap\ZZ^{m+n}$$ and$$\ab''=(a_1'', \ldots, a_{m+n}'')\in (N-2)\Pc$$for which $\ab=\ab'+\ab''$. According to Lemma \ref{compbips3} the following cases arise.

\medskip

\noindent
{\bf (Case 1)} Suppose that Lemma \ref{compbips2} ($\alpha$) holds. Then $\Pc$ is Gorenstein (\cite[Theorem 5.3]{HSF_5}). So, it is level.

\medskip

\noindent
{\bf (Case 2)} Suppose that Lemma \ref{compbips2} ($\beta$) holds. Let $a_1'=\cdots=a_{m'}=1$ and $a_i'=\min\{2c_i-1,a_i\}$ for each $i\in [m+n]\setminus [m]$. Furthermore, set $a_i'':=a_i-a_i'$ for each $i\in [m+n]$.

\medskip

\noindent
{\bf (Case 3)} Suppose that Lemma \ref{compbips2} ($\gamma$) holds. It follows from ($\gamma$) (i) that $c_{m+1}+\cdots +c_{m+n}\geq 2$. If $c_{m+1}+\cdots +c_{m+n}=2$, then $m=1$. This contradicts the inequalities  $c_1+\cdots +c_m>c_{m+1}+\cdots +c_{m+n}$ and $c_1\leq c_{m+1}+\cdots +c_{m+n}$.  Let $c_{m+1}+\cdots +c_{m+n}\geq 3$. Set$$X_1:=\{i\in [m] : c_i=2, a_i=2N-1\}, \quad X_2:=\{i\in [m] : c_i=2, a_i=2N-2\}.$$Since $a_1+\cdots +a_m\leq N(c_{m+1}+\cdots +c_{m+n})-1$, it follows that
\begin{eqnarray*}
&& (2N-1)|X_1|+(2N-2)|X_2|+(m-|X_1|-|X_2|) \\ &\leq & N(c_{m+1}+\cdots +c_{m+n})-1.    \end{eqnarray*}
This inequality together with ($\gamma$) (i) yields that
\begin{eqnarray*}
&& (N-1)(2|X_1|+|X_2|)\leq (2N-2)|X_1|+(2N-3)|X_2|\\  & \leq &  (N-1)(c_{m+1}+\cdots +c_{m+n}).    
\end{eqnarray*}
Hence $2|X_1|+|X_2|\leq c_{m+1}+\cdots +c_{m+n}$. Together with ($\gamma$) (i), this inequality implies that
\begin{eqnarray*}
& & 3|X_1|+2|X_2|+(m-|X_1|-|X_2|) \\ & \leq & 2(c_{m+1}+\cdots +c_{m+n})-1.
\end{eqnarray*}
Following the above inequality, one may choose integers $a_1',\ldots, a_m'$ for which

$\bullet$ $1\leq a_i'\leq \min\{a_i, 2c_i-1\}$ for each $i\in [m]$;

$\bullet$ $a_i'=3$, for each $i\in X_1$;

$\bullet$ $2\leq a_i'\leq 3$, for each $i\in X_2$;

$\bullet$ $a_1'+\cdots +a_m'\leq 2(c_{m+1}+\cdots +c_{m+n})-1$; 

$\bullet$ $a_1'+\cdots +a_m'$ is as large as possible. In other words, at least one of the following conditions holds. 

\begin{itemize}
\item [(1)] $a_1'+\cdots +a_m'=2(c_{m+1}+\cdots +c_{m+n})-1$;
\item [(2)]  $a_i'=\min\{a_i, 2c_i-1\}$ for every integer $i\in [m]$.
\end{itemize}
\noindent
For each $i\in [m+n]\setminus [m]$, set $a_i':=\min\{2c_i-1,a_i\}$. Then $$\ab'=(a_1', \ldots, a_{m+n}')\in 2(\Pc\setminus\partial\Pc)\cap\ZZ^{m+n}.$$ 
For each $i\in [m+n]$, set $a_i'':=a_i-a_i'$ and put $\ab'':=(a_1'', \ldots, a_{m+n}'')$. To complete our proof, we show that $\ab''\in (N-2)\Pc$. We have

$\bullet$ $a_i''\geq 0$ for each $i\in [m+n]$; 

$\bullet$ $a_i''\leq (N-2)c_i$ for each $i=[m]$ with $c_i=2$;

$\bullet$ $a_i''\leq (N-2)c_i$ for each $i\in [m+n]\setminus [m]$.

\noindent
So, we only need to show that$$a_1''+\cdots +a_m''\leq (N-2)(c_{m+1}+\cdots +c_{m+n}).$$
If (1) holds, then
\begin{eqnarray*}
&& a_1''+\cdots +a_m''\\&=&(a_1+\cdots +a_m)-(a_1'+\cdots +a_m')\\  &\leq &N(c_{m+1}+\cdots +c_{m+n})-1-\big(2(c_{m+1}+\cdots +c_{m+n})-1\big)\\ &=&(N-2)(c_{m+1}+\cdots +c_{m+n}).   
\end{eqnarray*}
Assume that (1) does not hold. Suppose $a_i''\neq 0$ for some $i\in [m]$. Then $a_i'=2c_i-1$. Since (1) does not hold, we deduce that $c_i\neq c_{m+1}+\cdots +c_{m+n}$. In other words, $c_i=2$. Set$$X_3:=\{i\in [m]: c_i=2,  a_i\geq 4\}.$$ For each $i\in [m]$, one has $a_i''>0$ if and only if $i\in X_3$. Furthermore,
\begin{eqnarray*}
a_1''+\cdots +a_m''&=&\sum_{i\in X_3}a_i''=\sum_{i\in X_3}(a_i-3)\leq\sum_{i\in X_3}(Nc_i-4)\\  &=&\sum_{i\in X_3}(2N-4)=(2N-4)|X_3|.
\end{eqnarray*}
If $|X_3|\leq (c_{m+1}+\cdots +c_{m+n})/2$, then the above inequality implies that$$a_1''+\cdots +a_m''\leq (N-2)(c_{m+1}+\cdots +c_{m+n}),$$as required. Let $|X_3|> (c_{m+1}+\cdots +c_{m+n})/2$. Since $a_i'\geq 1$ for each $i\in [m]$, we deduce that
\begin{eqnarray*}
 a_1'+\cdots +a_m'&\geq &(m-|X_3|)+\sum_{i\in X_3}a_i'= m-|X_3|+3|X_3|=m+2|X_3|\\  &\geq &(c_{m+1}+\cdots +c_{m+n}-1)+2(c_{m+1}+\cdots +c_{m+n}+1)/2\\ & = &2(c_{m+1}+\cdots +c_{m+n}),
\end{eqnarray*}
a contradiction.

\medskip

\noindent
{\bf (Case 4)} Suppose that Lemma \ref{compbips2} ($\delta$) holds. If $c_{m+1}+\cdots +c_{m+n}\leq 2$, then one easily verifies that $\Pc$ is level. Let $c_{m+1}+\cdots +c_{m+n}\geq 3$.  Let $X_1$ and $X_2$ be the sets defined in Case 3. It follows from ($\delta$) (iv) that
\begin{eqnarray*}
2\big |\{i\in [m] :c_i=2\}\big | &\leq & 2m-3m+2(c_{m+1}+\cdots +c_{m+2})-1\\ & = & 2(c_{m+1}+\cdots +c_{m+n})-m-1.
\end{eqnarray*}
Hence, 
\begin{eqnarray*}
& &3|X_1|+2|X_2|+(m-|X_1|-|X_2|)\\ & =& 2|X_1|+|X_2|+m\\ & \leq & 2\big |\{i\in [m] :c_i=2\}\big |+m \leq 2(c_{m+1}+\cdots +c_{m+n})-1.
\end{eqnarray*}
As in the proof of Case 3, the above inequality implies that there are integers $a_1',\ldots, a_m'$ for which

$\bullet$ $1\leq a_i'\leq \min\{a_i, 2c_i-1\}$ for each $i\in [m]$;

$\bullet$ $a_i'=3$ for each $i\in X_1$;

$\bullet$ $2\leq a_i'\leq 3$ for each $i\in X_2$;

$\bullet$ $a_1'+\cdots +a_m'\leq 2(c_{m+1}+\cdots +c_{m+n})-1$; 

$\bullet$ At least one of the following conditions holds:

\begin{itemize}
\item [(1)] $a_1'+\cdots +a_m'=2(c_{m+1}+\cdots +c_{m+n})-1$;
\item [(2)]  $a_i'=\min\{a_i, 2c_i-1\}$, for every integer $i\in [m]$.
\end{itemize}

\noindent
For each $i\in [m+n]\setminus [m]$, set $a_i':=\min\{2c_i-1,a_i\}$. One has $$\ab'=(a_1', \ldots, a_{m+n}')\in 2(\Pc\setminus\partial\Pc)\cap\ZZ^{m+n}.$$ 
For each $i\in [m+n]$, set $a_i'':=a_i-a_i'$ and put $\ab'':=(a_1'', \ldots, a_{m+n}'')$. As in the proof of Case 3, it is enough to show that$$a_1''+\cdots +a_m''\leq (N-2)(c_{m+1}+\cdots +c_{m+n}).$$
If (1) holds, then the argument is similar to that of Case 3. Assume that (1) does not hold. Let $a_i''\neq 0$ for some $i\in [m]$. Then $a_i'=2c_i-1$. Since (1) does not hold, we deduce that $c_i\neq c_{m+1}+\cdots +c_{m+n}$. Thus $c_i\in\{1,2\}$. Set $X_3:=X_4\sqcup X_5$, where$$X_4:=\{i\in [m]: c_i=2,  a_i\geq 4\}, \quad X_5=\{i\in [m]: c_i=1, a_i\geq 2\}.$$ For each $i\in [m]$, one has $a_i''>0$ if and only if $i\in X_3$. Furthermore,
\begin{eqnarray*}
a_1''+\cdots +a_m''& =& \sum_{i\in X_3}a_i''=\sum_{i\in X_4}(a_i-3)+\sum_{i\in X_5}(a_i-1)\\ &\leq& \sum_{i\in X_4}(Nc_i-4)+\sum_{i\in X_5}(Nc_i-2)\\  &=& \sum_{i\in X_4}(2N-4)+\sum_{i\in X_5}(N-2)\\  &=& (2N-4)|X_4|+(N-2)|X_5|\\  &=& (N-2)(2|X_4|+|X_5|).
\end{eqnarray*}
If $2|X_4|+|X_5|\leq (c_{m+1}+\cdots +c_{m+n})$, then the above inequality implies that$$a_1''+\cdots +a_m''\leq (N-2)(c_{m+1}+\cdots +c_{m+n}).$$
If $2|X_4|+|X_5|> (c_{m+1}+\cdots +c_{m+n})$, then, since $a_i'\geq 1$ for each $i\in [m]$, one has
\begin{eqnarray*}
 &&a_1'+\cdots +a_m'
\\ &\geq &(m-|X_4|-|X_5|)+\sum_{i\in X_4}a_i'+\sum_{i\in X_5}a_i'
\\  &= &m-|X_4|-|X_5|+3|X_4|+|X_5|=m+2|X_4|
\\  &> &m+ (c_{m+1}+\cdots +c_{m+n})-|X_5|
\\ &\geq &2(c_{m+1}+\cdots +c_{m+n})-1,
\end{eqnarray*}
where the last inequality follows from ($\delta$) (i). This is a contradiction.
\, \, \, \, \, 
\end{proof}

\begin{Lemma} \label{compbips1} 
Set$$s:=\max\bigg\{\Big\lfloor\, \frac{m}{c_{m+1}+\cdots +c_{m+n}}\, \Big\rfloor+1, \, \max\Big\{\Big\lfloor\frac{5-c_i}{2}\Big\rfloor: i\in [m+n]\Big\}\bigg\}$$and $$A:=\{i\in [m] : c_i=c_{m+1}+\cdots +c_{m+n}\}.$$If $s=1$, then $\conv(\Dc(K_{m,n},\mathfrak{c}))$ is nearly Gorenstein if and only if one of the following conditions holds.

\begin{itemize}
    \item [($\alpha$)] {\em (i)} $m=2n-1$, {\em (ii)} $c_{m+1}=\cdots =c_{m+n}=2$ and {\em (iii)} $c_i\in \{2,2n\}$ for each $i\in [m]$;
    \item [($\beta$)] There do not exist disjoint subsets $X_1, X_2, X_3, X_4, X_5\subseteq [m]$ for which the following conditions hold:
    \begin{itemize}
        \item [(i)]
        \begin{align*}
        \, \, \, & 2m+\sum_{i\in X_1\sqcup X_2}c_i
        \\\leq \, \, \, & c_{m+1}+\cdots +c_{m+n}+2|X_1|+3|X_2|+|X_3|+2|X_4|+|X_5|;
        \end{align*}
        \item [(ii)]
        \begin{align*}
     \, \, \, &      2|X_1|+|X_2|+3|X_3|+2|X_4|+3|X_5|+\sum_{i\in [m]\setminus (X_3\sqcup X_4\sqcup X_5)}c_i\\  \geq \, \, \, & 2m +c_{m+1}+\cdots +c_{m+n};  
        \end{align*}
        \item [(iii)] $c_i\geq 3$ for each $i\in X_2\sqcup X_3$;
        \item [(iv)] $c_i=2$ for each $i\in X_5$;
        \item [(v)] $(X_1\sqcup X_2)\cap A=\emptyset$;
        \item [(vi)] $c_i\geq 4$ for each $i\in [m]\setminus(X_1\sqcup\cdots\sqcup X_5)$;
        \item [(vii)] Either $2|X_1|+|X_2|+|X_5|\geq m+2$ or the inequalities$$|X_3|+2|X_4|+|X_5|\geq m$$ and
        \begin{align*}
         & 2|X_1|+|X_2|+4|X_3|+4|X_4|+4|X_5|+4|A'|\\  \geq \, \, \, &  3m+(c_{m+1}+\cdots +c_{m+n})-\sum_{i\in [m]\setminus (X_3\cup X_4\cup X_5\cup A)}c_i   
        \end{align*}
        hold, where $A':=A\setminus (X_1\sqcup\cdots \sqcup X_5)$.            \end{itemize}
\end{itemize}  
\end{Lemma}

\begin{proof}
Set $\Pc:=\conv(\Dc(G,\mathfrak{c}))$ and and recall Lemma \ref{Osaka}.  It follows from the definition of $s$ together with $s=1$ that $(\Pc\setminus\partial\Pc)\cap\ZZ^n\neq\emptyset$. Furthermore, $c_i\geq 2$, for each $i\in [m+n]$.

\medskip

\noindent

{\bf (``If'')}
If ($\alpha$) holds, then \cite[Theorem 5.3]{HSF_5} says that $\Pc$ is Gorenstein. So, let ($\beta$) hold. We prove that $\Pc$ is nearly Gorenstein. Let $\ab=(a_1, \ldots, a_{m+n})$ belong to $\Pc\cap\ZZ^{m+n}$. Using induction on$$(c_{m+1}+\cdots c_{m+1})-(a_1+\cdots +a_m),$$we show that $\xb^\ab t\in \omega_{K[\Pc]}.\omega_{K[\Pc]}^{-1}$.

First, assume that $a_1+\cdots +a_m=c_{m+1}+\cdots +c_{m+n}$. If $c_{m+1}+\cdots + c_{m+n}=2$, then $n=1$ and $c_{m+1}=2$. Since $c_i\leq c_{m+1}+\cdots +c_{m+n}$ for each $i\in [m]$, we deduce that $c_1=\cdots =c_m=2$. Furthermore, it follows from $c_1+\cdots +c_m > c_{m+1}+\cdots +c_{m+n}$ that $m\geq 2$. This contradicts $s=1$. Thus $c_{m+1}+\cdots + c_{m+n}\geq 3$.

If there is $i\in A$ for which $a_i=c_i$, then $a_j=0$ for each $j\in [m]\setminus\{i\}$. In this case, $\xb^{\ab}t$ can be written as the product of $$\big(x_i^{c_i-m-1}(x_1\cdots x_mx_{m+1}^{\max\{1,a_{m+1}-1\}}\cdots x_{m+n}^{\max\{1,a_{m+n}-1\}}t\big)$$ and$$\frac{\xb^{\ab}}{\big(x_i^{c_i-m-1}(x_1\cdots x_mx_{m+1}^{\max\{1,a_{m+1}-1\}}\cdots x_{m+n}^{\max\{1,a_{m+n}-1\}}\big)}.$$So, $\xb^\ab t\in \omega_{K[\Pc]}.\omega_{K[\Pc]}^{-1}$. 

If there is $i\in A$ for which $a_i=c_i-1$, then there is $i'\in [m]\setminus \{i\}$ with $a_{i'}=1$. Furthermore, $a_j=0$ for each $j\in [m]\setminus\{i,i'\}$. Then the same argument as above shows that $\xb^\ab t\in \omega_{K[\Pc]}.\omega_{K[\Pc]}^{-1}$. 

So, we assume that for each $i\in A$, one has $a_i\leq c_i-2$.  Set 
\begin{eqnarray*}
    X_1&:=&\{i\in [m] : a_i=c_i\}, 
    \\X_2&:=&\{i\in [m] : a_i=c_i-1, \,  c_i\geq 3\},
    \\X_3&:=&\{a_i=1, c_i\geq 3\},
    \\X_4&:=&\{i\in [m] : a_i=0\},
    \\X_5&:=&\{i\in [m] : a_i=1, \, c_i=2\}.
\end{eqnarray*}
Then $X_1, \cdots, X_5$ are disjoint subsets of $[m]$. Obviously, ($\beta$) (iii), (iv) and (v)  hold for $X_1, \ldots, X_5$. On the other hand, for each  $i\in [m]\setminus (X_1\cup\cdots \cup X_5)$, one has $2\leq a_i\leq c_i-2$. In particular, ($\beta$) (vi) holds. Since$$a_1+\cdots +a_m=c_{m+1}+\cdots +c_{m+n},$$we conclude that

\begin{eqnarray*}
&& 2(m-|X_1|-|X_2|-|X_3|-|X_4|-|X_5|)
\\ &&\quad \qquad +\,\sum_{i\in X_1}c_i+\sum_{i\in X_2}(c_i-1)+|X_3|+|X_5|\\ &\leq & c_{m+1}+\cdots +c_{m+n}\\ &\leq & \sum_{i\in [m]\setminus (X_i\sqcup\cdots\sqcup X_5)}(c_i-2)+\sum_{i\in X_1}c_i+\sum_{i\in X_2}(c_i-1)+|X_3|+|X_5|. 
\end{eqnarray*}
The above inequalities show that ($\beta$) (i) and (ii) hold for $X_1, \ldots, X_5$. As a consequence, $X_1, \ldots, X_5$ do not satisfy ($\beta$) (vii). In other words, $$2|X_1|+|X_2|+|X_5|\leq m+1$$ and either $|X_3|+2|X_4|+|X_5|\leq m-1$ or
\begin{eqnarray*}
         && 2|X_1|+|X_2|+4|X_3|+4|X_4|+4|X_5|+4|A'|\\  &\leq &  3m-1+(c_{m+1}+\cdots +c_{m+n})-\sum_{i\in [m]\setminus (X_3\cup X_4\cup X_5\cup A)}c_i.  
        \end{eqnarray*}
If $$|X_3|+2|X_4|+|X_5|\leq m-1,$$ then
\begin{eqnarray*}
&& |X_1|+|X_2|-|X_4|+(m-
\sum_{i=1}^{5}|X_i|
-|A'|) +\sum_{i\in A'}(a_i-1)\\ & \geq & |X_1|+|X_2|-|X_4|+(m-\sum_{i=1}^{5}|X_i|-|A'|)+|A'|\geq 1.    
\end{eqnarray*}
If
\begin{eqnarray*}
&&
2|X_1|+|X_2|+4|X_3|+4|X_4|+4|X_5|+4|A'|\\  & \leq & 3m-1+(c_{m+1}+\cdots +c_{m+n})-\sum_{i\in [m]\setminus (X_3\cup X_4\cup X_5\cup A)}c_i,  
\end{eqnarray*}
then a routine computation using $a_1+\cdots +a_m=c_{m+1}+\cdots c_{m+n}$ says that  

\begin{eqnarray*}
&& |X_1|+|X_2|-|X_4|+(m-\sum_{i=1}^{5}|X_i|-|A'|)+\sum_{i\in A'}(a_i-1)
\\ & \geq& m-|X_3|-2|X_4|-|X_5|-2|A'|+(c_{m+1}+\cdots +c_{m+n})\\ 
& -&\sum_{i\in X_1}c_i-\sum_{i\in X_2}(c_i-1)-|X_3|-|X_5| -\sum_{i\in [m]\setminus (X_1\cup\cdots\cup X_5\cup A)}(c_i-2)\\ 
& =&3m-2|X_1|-|X_2|-4|X_3|-4|X_4|-4|X_5|-4|A'|+(c_{m+1}+\cdots +c_{m+n})\\ &&\, \, \, \, \, \, \, \, \, \, \, \, \, \, \, \,  -\sum_{i\in [m]\setminus (X_3\cup X_4\cup X_5\cup A)}c_i \\&\geq& 1.
\end{eqnarray*}

\noindent
It follows from the above argument together with $$2|X_1|+|X_2|+|X_5|\leq m+1$$ that 
\begin{eqnarray*}
&& -(m-\sum_{i=1}^{5}|X_i|)+|X_1|-|X_3|-|X_4|
\\ 
&\leq& 1 
\\ 
&\leq&  |X_1|+|X_2|-|X_4|+(m-\sum_{i=1}^{5}|X_i|-|A'|)+\sum_{i\in A'}(a_i-1).
\end{eqnarray*}
These inequalities imply that there are integers $a_1', \ldots, a_m'$ with $a_1'+\cdots + a_m'=1$ for which 

$\bullet$ $a_i'=1$ for each $i\in X_1$;

$\bullet$ $0\leq a_i'\leq 1$ for each $i\in X_2$;

$\bullet$ $-1\leq a_i'\leq 0$ for each $i\in X_3$;

$\bullet$ $a_i'=-1$ for each $i\in X_4$;

$\bullet$ $a_i'=0$ for each $i\in X_5$;

$\bullet$ $-1\leq a_i'\leq 1$ for each $i\in [m]\setminus (X_1\sqcup\cdots \sqcup X_5\sqcup A')$; 

$\bullet$ $-1\leq a_i'\leq a_i-1$ for each $i\in A'$.

\noindent
Putting $a_i'=\min\{1, a_i-1\}$ for each $i\in [m+n]\setminus [m]$ and defining $\ab':=(a_1', \ldots, a_{m+n}')$, one has $$\xb^\ab t=(\xb^{\ab-\ab'}t)(\xb^{\ab'}).$$By the choice of $\ab'$, it follows that $\xb^{\ab'}\in \omega_{K[\Pc]}^{-1}$. We prove $\xb^{\ab-\ab'}t\in \omega_{K[\Pc]}$. Obviously, each component of $\ab-\ab'$ is positive. Furthermore, the definition of $\ab'$ shows that for each $i\in [m+n]$, one has $a_i-a_i'\leq c_i-1$. On the other hand,
\begin{eqnarray*}
&& (a_1-a_1')+\cdots +(a_m-a_m')
\\&=&(a_1+\cdots +a_m)-(a_1'+\cdots +a_m')\\ & =&(c_{m+1}+\cdots + c_{m+n})-1,  
\end{eqnarray*}
where the last equality follows from $a_1'+\cdots +a_m'=1$ and $$a_1+\cdots +a_m=c_{m+1}+\cdots +c_{m+n}.$$ Thus $\xb^{\ab-\ab'}t\in \omega_{K[\Pc]}$. This shows that $\xb^\ab t\in \omega_{K[\Pc]}.\omega_{K[\Pc]}^{-1}$, as required.

Second, assume that $a_1+\cdots +a_m < c_{m+1}+\cdots +c_{m+n}$. Then there is $i\in [m]$ with $x_i\xb^{\ab}t\in K[\Pc]$. Using induction hypothesis, there are $\xb^{\ab_1}t^{\ell_1}\in \omega_{K[\Pc]}$ together with $\xb^{\ab_2}t^{\ell_2}\in \omega_{K[\Pc]}^{-1}$ for which$$x_i\xb^{\ab}t=(\xb^{\ab_1}t^{\ell_1})(\xb^{\ab_2}t^{\ell_2}).$$If the $i$th component of $\ab_2$ is nonnegative, then $\xb^{\ab_2}t^{\ell_2}/x_i\in \omega_{K[\Pc]}^{-1}$. In this case$$\xb^{\ab}t=(\xb^{\ab_1}t^{\ell_1})(\xb^{\ab_2}t^{\ell_2}/x_i)\in \omega_{K[\Pc]}.\omega_{K[\Pc]}^{-1},$$and we are done. So, suppose that the $i$th component of $\ab_2$ is negative. It follow from $x_i\xb^{\ab}=\xb^{\ab_1}\xb^{\ab_2}$ that the $i$th component of $\ab_1$ is at least two. Then $(\xb^{\ab_1}t^{\ell_1})/x_i\in \omega_{K[\Pc]}$ and a similar argument as above yields that $\xb^{\ab}t\in \omega_{K[\Pc]}.\omega_{K[\Pc]}^{-1}$.

\medskip

\noindent

{\bf (``Only If'')}
Suppose that $\Pc$ is nearly Gorenstein. We show that ($\alpha$) or ($\beta$) holds. If $|(P\setminus \partial\Pc)\cap \ZZ^{m+n}|=1$. then $\Pc$ is pseudo-Gorenstein. In this case, we deduce from \cite{Miyashita} that $\Pc$ is Gorenstein. Hence, ($\alpha$) holds (\cite[Theorem 5.3]{HSF_5}). Let $|(\Pc\setminus \partial\Pc)\cap\ZZ^{m+n}|\geq 2$. We prove that ($\beta$) holds.

Let $\bb=(b_1, \ldots, b_{m+n})\in \Pc\cap \ZZ^{m+n}$ satisfy $$b_1+\cdots +b_m=c_{m+1}+\cdots + c_{m+n}.$$Since $$c_1+\cdots +c_m>c_{m+1}+\cdots c_{m+n},$$such a point $\bb$ exists. As $\Pc$ is nearly Gorenstein, one has $\xb^{\bb}t\in \omega_{K[\Pc]}.\omega_{K[\Pc]}^{-1}$. Therefore, there are $\xb^{\bb'}t^p\in \omega_{K[\Pc]}$ and $\xb^{\bb''}t^{q}\omega_{K[\Pc]}^{-1}$ for which $$\xb^{\bb}t=(\xb^{\bb'}t^{p})(\xb^{\bb''}t^{q})=\xb^{\bb'+\bb''}t^{p+q}.$$Therefore, $\bb'+\bb''=\bb$ and $p+q=1$. Since $x_1\cdots x_{m+n}t\in \omega_{K[\Pc]}$ and $\xb^{\bb''}t^{q}\in \omega_{K[\Pc]}^{-1}$, it follows that
\begin{eqnarray}
\label{eeeeeeeeeee}   
(x_1\cdots x_{m+n}t)(\xb^{\bb''}t^{q})\in K[\Pc].
\end{eqnarray}
In particular, $q\geq -1$. Let $q=-1$. Then it follows from (\ref{eeeeeeeeeee}) that $\xb^{\bb''}=x_1^{-1}\cdots x_{m+n}^{-1}$. We deduce from $|(\Pc\setminus\partial\Pc)\cap\ZZ^{m+n}|\geq 2$ that $x_i(x_1, \cdots x_{m+n})t\in \omega_{K[\Pc]}$ for some $i\in [m+n]$. Hence, it follows from $\xb^{\bb''}t^{q}\in\omega_{K[\Pc]}^{-1}$ that
\begin{align*}
& x_i=x_i(x_1\cdots x_{m+n}t)(x_1^{-1}\cdots x_{m+n}^{-1}t^{-1})=x_i(x_1x_2\cdots x_{m+n}t)(\xb^{\bb''}t^{q})\in K[\Pc],
\end{align*}
a contradiction. Thus $q\geq 0$. Since $p\geq 1$ and $p+q=1$, we conclude that $$(p,q)=(1,0).$$
Since $c_1+\cdots +c_m> c_{m+1}+\cdots +c_{m+n}$, there is an integer $N\geq 1$ for which$$\sum_{i=1}^m(Nc_i-1)\geq N(c_{m+1}+\cdots +c_{m+n})-1.$$Hence, there is $\db=(d_1, \ldots, d_{m+n})\in N(\Pc\setminus\partial\Pc) \cap \ZZ^{m+n}$ with $$d_1+\cdots +d_m=N(c_{m+1}+\cdots +c_{m+n})-1.$$In particular, $\xb^{\db}t^N\in\omega_{K[\Pc]}$. Since 
$\xb^{\bb''}\in\omega_{K[\Pc]}^{-1}$, we deduce that $\bb''+\db\in N\Pc$. Therefore,
\begin{eqnarray}
\label{gggggg}
b_1''+\cdots +b_m''\leq N(c_{m+1}+\cdots +c_{m+n})-(d_1+\cdots +d_m)=1.
\end{eqnarray}

For each $i\in [m]\setminus A$, we have $c_i< c_{m+1}+\cdots +c_{m+n}$. So,  there is an integer $N'\geq 1$ for which$$(N'c_i-1)+m\leq N'(c_{m+1}+\cdots +c_{m+n}).$$In particular, $x_i^{N'c_i-2}(x_1\cdots x_{m+n})t^{N'}\in \omega_{K[\Pc]}$. It follows from $\xb^{\bb''}\in\omega_{K[\Pc]}^{-1}$ that$$x_i^{Nc_i-2}(x_1\cdots x_{m+n})\xb^{\bb''}t^{N'}\in {K[\Pc]}.$$Hence $b_i''\leq 1$.

Everything we proved so far is valid for each $\bb=(b_1, \ldots, b_{m+n})\in \Pc\cap \ZZ^{m+n}$ with$$b_1+\cdots +b_m=c_{m+1}+\cdots +c_{m+n}.$$In what follows, we choose an appropriate $\bb\in \Pc\cap \ZZ^{m+n}$ satisfying the above equality. Recall that our goal is to show that ($\beta$) holds. By contraction, assume that ($\beta$) does not hold. Then there are disjoint subsets $X_1, \ldots X_5$ of $[m]$ satisfying ($\beta$) (i) -- (vii). Set

$\bullet$ $b_i=c_i$ for each $i\in X_1$;

$\bullet$ $b_i=c_i-1$ for each $i\in X_2$;

$\bullet$ $b_i=1$ for each $i\in X_3$;

$\bullet$ $b_i=0$ for each $i\in X_4$;

$\bullet$ $b_i=1$ for each $i\in X_5$; and

\noindent
Using ($\beta$) (i) and (ii), one easily sees that  
\begin{eqnarray*}
&& 2(m-\sum_{i=1}^{5} |X_i|)+\sum_{i\in X_1\sqcup\cdots\sqcup X_5}b_i \\ &\leq& c_{m+1}+\cdots +c_{m+n}\\ & \leq & \sum_{i\in [m]\setminus (X_1\sqcup\cdots\sqcup X_5)}(c_i-2)+\sum_{i\in X_1\sqcup\cdots\sqcup X_5}b_i.  
\end{eqnarray*}
The above inequalities together with ($\beta$) (vi) imply that for each $$i\in [m]\setminus (X_1\sqcup\cdots\sqcup X_5),$$ there is an integer $b_i$ with $2\leq b_i\leq c_i-2$ for which $$b_1+\cdots +b_m=c_{m+1}+\cdots +c_{m+n}.$$We choose$$(\, b_i \, )_{\,i \, \in \,[m]\setminus (X_1\sqcup\cdots\sqcup X_5)},$$for which $\sum_{i\in A'}b_i$ is as small as possible. In other words, if $b_k\geq 3$ for some $k\in A'$, then $b_i=c_i-2$ for each $i\in [m]\setminus (X_1\sqcup\cdots\sqcup X_5\sqcup A')$. Set $b_j:=0$ for each $j\in [m+n]\setminus [m]$.  Define $\bb:=(b_1, \ldots, b_{m+n})\in \Pc\cap\ZZ^{m+n}$. We know from (\ref{eeeeeeeeeee}) that each component of $\bb''$ is at least $-1$. Since $b_i'\leq c_i-1$ for each $i\in [m+n]$, we deduce that $b_i''\geq 1$ for each $i\in X_1$. Similarly, $b_i''\geq 0$ for each $i\in X_2\sqcup X_5$. Hence,
\begin{eqnarray*}
&& 2|X_1|+|X_2|+|X_5|-m\\ & =& |X_1|-|X_3|-|X_4|-(m-|X_1|-|X_2|-|X_3|-|X_4|-|X_5|)\\ & \leq & b_1''+\cdots +b_m''\leq 1,   
\end{eqnarray*}
where the last inequality follows from (\ref{gggggg}). Therefore, we conclude from ($\beta$) (vii) that
\begin{eqnarray}
\label{gggggggg}
|X_3|+2|X_4|+|X_5|\geq m
\end{eqnarray}
and
        \begin{eqnarray*}
         & &2|X_1|+|X_2|+4|X_3|+4|X_4|+4|X_5|+4|A'|\\ & \geq &3m+(c_{m+1}+\cdots +c_{m+n})-\sum_{i\in [m]\setminus (X_3\cup X_4\cup X_5\cup A)}c_i.   
        \end{eqnarray*}
The following two cases arise.

\medskip

\noindent
{\bf (Case 1)} Let $b_i=2$ for each $i\in A'$. It follows from $b_i'\geq 1$ that $b_i''\leq 1$ for each $i\in A'$. If $j\in A\setminus A'$, then by ($\beta$) (v) and the definition of $A'$, one has $j\in X_3\sqcup X_4\sqcup X_5$. Since $b_j'\geq 1$, we deduce that $b_j''\leq 0$. Furthermore, this argument shows that for each $j\in X_4$, one has $b_j''=-1$. In addition, we already know that $b_k''\leq 1$, for $k\in [m]\setminus A$. Thus
\begin{eqnarray*}
&& b_1''+\cdots +b_m''\\ &\leq& |X_1|+|X_2|-|X_4|+(m-|X_1|-|X_2|-|X_3|-|X_4|-|X_5|)\\ & = & m-|X_3|-2|X_4|-|X_5|\leq 0,    
\end{eqnarray*}
where the last inequality follows from (\ref{gggggggg}). Therefore,$$b_1'+\cdots +b_m'=(b_1+\cdots +b_m)-(b_1''+\cdots +b_m'')\geq c_{m+1}+\cdots + c_{m+n},$$a contradiction.

\medskip

\noindent
{\bf (Case 2)} Let $b_k\geq 3$ for some $k\in A'$. It follows from the choice of $\bb$ that $b_j=c_j-2$ for each $j\in [m]\setminus (X_1\sqcup\cdots\sqcup X_5\sqcup A')$. For each $i\in A'$, it follows from $b_i'\geq 1$ that $b_i''\leq b_i-1$. By a similar argument as in Case 1 together with $$b_1+\cdots +b_m=c_{m+1}+\cdots +c_{m+n}$$ and 
\begin{eqnarray*}
         & &2|X_1|+|X_2|+4|X_3|+4|X_4|+4|X_5|+4|A'|\\ & \geq &3m+(c_{m+1}+\cdots +c_{m+n})-\sum_{i\in [m]\setminus (X_3\cup X_4\cup X_5\cup A)}c_i,   
        \end{eqnarray*}
a routine computation shows that 
\begin{eqnarray*}
&&  b_1''+\cdots +b_m''
\\ & \leq & |X_1|+|X_2|-|X_4|+(m-\sum_{i=1}^{5}|X_i|-|A'|)+\sum_{i\in A'}(b_i-1)
\\ & =& m-|X_3|-2|X_4|-|X_5|-2|A'|+(c_{m+1}+\cdots +c_{m+n})-\sum_{i\in [m]\setminus A'}b_i
\\ & =& m-|X_3|-2|X_4|-|X_5|-2|A'|+(c_{m+1}+\cdots +c_{m+n})\\ 
& -&\sum_{i\in X_1}c_i-\sum_{i\in X_2}(c_i-1)-|X_3|-|X_5| -\sum_{i\in [m]\setminus (X_1\cup\cdots\cup X_5\cup A)}(c_i-2)
\\ &\leq & 0.    
\end{eqnarray*}
Finally the same argument as in Case 1 implies that$$b_1'+\cdots +b_m'\geq c_{m+1}+\cdots +c_{m+n},$$a contradiction.
\hspace{10.9cm}
\end{proof}

\begin{Lemma} \label{compbiplevs1}
Under the situation of Lemma \ref{compbips1}, if $\conv(\Dc(K_{m,n},\mathfrak{c}))$ is nearly Gorenstein, then it is level.   
\end{Lemma}

\begin{proof}
We follow the notations of Lemma \ref{compbips1} and its proof. Recall that $B=[m]\setminus A$. Suppose that $\Pc$ is not level. According to \cite[Theorem 2.5]{HSF_6}, the following cases arise:

\medskip

\noindent
{\bf (Case 1)} Assume that there is a subset $\emptyset \neq X\subset B$ for which 
\begin{eqnarray}
\label{bbb}
c_{m+1}+\cdots + c_{m+n}<\sum_{i\in X}c_i< c_{m+1}+\cdots + c_{m+n}-m+2|X|-1.  
\end{eqnarray}
Choose a subset $X$ of $B$ satisfying the above inequalities for which $|X|$ is minimum. It follows from the above inequalities that $2|X|\geq m+3$. Without loss of generality, let $X=[k]$, where $k$ is an integer with $1\leq k\leq m$. Since $c_k\geq 2$, the second inequality of (\ref{bbb}) shows that$$\sum_{i\in [k-1]}c_i< c_{m+1}+\cdots + c_{m+n}-m+2(k-1)-1.$$Hence, the minimality of $k=|X|$ yields that$$\sum_{i\in [k-1]}c_i\leq c_{m+1}+\cdots +c_{m+n}.$$This implies that there is integers $b_k$ with  $0\leq b_k\leq c_k$ for which $$c_1+\cdots +c_{k-1}+b_k=c_{m+1}+\cdots +c_{m+n}.$$In particular, $x_1^{c_1}\cdots x_{k-1}^{c_{k-1}}x_k^{b_k}t\in K[\Pc]$. Since $\Pc$ is nearly Gorenstein, we deduce from the proof of Lemma \ref{compbips1} that there are $\xb^{\bb'}t\in\omega_{K[\Pc]}$ and $\xb^{\bb''}\in \omega_{K[\Pc]}^{-1}$ for which $x_1^{c_1}\cdots x_{k-1}^{c_{k-1}}x_k^{b_k}=\xb^{\bb'+\bb''}$.  Set $\bb'=(b_1', \ldots, b_{m+n}')$ and $\bb''=(b_1'', \ldots, b_{m+n}'')$. We know from the proof of Lemma \ref{compbips1} that $b_k''\geq -1$. Hence $b_k'\leq b_k+1$. The same argument shows that $b_{k+1}'=\cdots = b_m'=1$. Then 
\begin{eqnarray*}
&& b_1'+\cdots +b_m' \\ &\leq & (c_1-1)+\cdots +(c_{k-1}-1)+(b_k+1)+m-k\\ & =&(c_{m+1}+\cdots +c_{m+n})-(k-1)+1+m-k\\ & =&(c_{m+1}+\cdots +c_{m+n})+m-2k+2.    
\end{eqnarray*}
Since $x_1^{c_1}\cdots x_{k-1}^{c_{k-1}}x_k^{b_k}=\xb^{\bb'+\bb''}$ and $$c_1+\cdots +c_{k-1}+b_k=c_{m+1}+\cdots +c_{m+n},$$ we conclude that
\begin{align*}
& b_1''+\cdots +b_m''\geq 2k-m-2=2|X|-m-2\geq 1.    
\end{align*}
On the other hand, we know from inequality (\ref{gggggg}) in the proof of Lemma \ref{compbips1} that $b_1''+\cdots +b_m''\leq 1$. Hence, the above inequality shows that $2k=m+3$ and $b_k'=b_k+1$. In particular, $b_k\leq c_k-2$. Therefore, $$\sum_{i\in [k]}c_i\geq c_{m+1}+\cdots +c_{m+n}+2.$$Then the second inequality of (\ref{bbb}) implies that $-m+2k-1\geq 3$, contradicting $2k=m+3$.

\medskip

\noindent
{\bf (Case 2)} Assume that there is a subset $\emptyset \neq X\subset [m]$ for which 
\begin{eqnarray}
\label{abcdefg}
\sum_{i=1}^{m}c_i-\sum_{i\in X}c_i < c_{m+1}+\cdots + c_{m+n}\leq 2|X|+\sum_{i=1}^{m}c_i-\sum_{i\in X}c_i-m. 
\end{eqnarray}
As in Case 1, we choose a nonempty subset $X$ of $[m]$ satisfying the above inequality for which $|X|$ is minimum. Since $c_j\geq 2$, we deduce from the above inequalities that for each $j\in X$, the inequality$$c_{m+1}+\cdots +c_{m+n}\leq 2|X\setminus\{j\}|+\sum_{i=1}^{m}c_i-\sum_{i\in X\setminus\{j\}}c_i-m$$holds. As a consequence, the minimality of $|X|$ implies that$$\sum_{i=1}^{m}c_i-\sum_{i\in X\setminus \{j\}}c_i \geq c_{m+1}+\cdots + c_{m+n},$$for each $j\in X$.  Set $X'=[m]\setminus X$. Let, say, $X'=[\ell]$, where $\ell$ is an integer with $1\leq \ell\leq m-1$. Then the first inequality of  (\ref{abcdefg}) and the above argument imply that$$\sum_{i=1}^{\ell}c_i<c_{m+1}+\cdots +c_{m+n}\leq \sum_{i=1}^{\ell+1}c_i.$$Hence, there is an integer $b_{\ell+1}$ with $1\leq b_{\ell+1}\leq c_{\ell+1}$ for which$$c_1+\cdots +c_{\ell}+b_{\ell+1}=c_{m+1}+\cdots +c_{m+n}.$$In particular, $x_1^{c_1}\cdots x_{\ell}^{c_{\ell}}x_{\ell+1}^{b_{\ell+1}}t\in K[\Pc]$. Since $\Pc$ is nearly Gorenstein, we deduce from the proof of Lemma \ref{compbips1} that there are $\xb^{\bb'}t\in\omega_{K[\Pc]}$ and $\xb^{\bb''}\in \omega_{K[\Pc]}^{-1}$ for which $x_1^{c_1}\cdots x_{\ell}^{c_{\ell}}x_{\ell+1}^{b_{\ell+1}}=\xb^{\bb'+\bb''}$.  Set $\bb'=(b_1', \ldots, b_{m+n}')$ and $\bb''=(b_1'', \ldots, b_{m+n}'')$. It follows from $c_1+\cdots +c_{\ell}< c_{m+1}+\cdots +c_{m+n}$ that $c_i\neq c_{m+1}+\cdots +c_{m+n}$, for each $i\in [\ell]$. Hence, the proof of Lemma \ref{compbips1} says that $b_1'', \ldots, b_{\ell}''$ are at most one. This yields that $b_i'\geq c_i-1$ for each $i\in [\ell]$. Since $b_{\ell+1}', \ldots, b_m'$ are positive, one has
\begin{eqnarray*}
&& b_1'+\cdots +b_m' \\ &\geq& \sum_{i=1}^{\ell}(c_i-1)+(m-\ell) \\ &=& \sum_{i=1}^{\ell}c_i+m-2\ell \\ &\geq& c_{m+1}+\cdots +c_{m+n},    
\end{eqnarray*}
where the last inequality follows from the second inequality of  (\ref{abcdefg}). This is a contradiction.
\hspace{1.4cm}
\end{proof}

Summarizing Lemmas \ref{compbips3}, \ref{compbips2} and \ref{compbips1}, we obtain the following theorem.

\begin{Theorem} \label{compbipmain}
Set$$s:=\max\bigg\{\Big\lfloor\frac{m}{c_{m+1}+\cdots +c_{m+n}}\Big\rfloor+1, \, \max\Big\{\Big\lfloor\frac{5-c_i}{2}\Big\rfloor: i\in [m+n]\Big\}\bigg\}.$$If $|\Bc(K_{m,n},\mathfrak{c})|\geq 2$, then 
$\conv(\Dc(K_{m,n},\mathfrak{c}))$ is nearly Gorenstein if and only if one of the following conditions holds.

\begin{itemize}
    \item [(1)] $s=3$ and conditions (ii) and (iii) of Lemma \ref{compbips3} hold.
    
    \item [(2)] $s=2$ and one of the conditions ($\alpha$), ($\beta$), ($\gamma$) and ($\delta$) of Lemma \ref{compbips2} holds.
    
    \item [(3)] $s=1$ and one of the conditions ($\alpha$) and ($\beta$) of Lemma \ref{compbips1} holds.
   \end{itemize} 
\end{Theorem}

The following corollary is an immediate consequence of Lemmas \ref{compbiplevs3}, \ref{compbiplevs2} and \ref{compbiplevs1}.

\begin{Corollary}
\label{final}
Every nearly Gorenstein $K_{m,n}$-polytope is level.   
\end{Corollary}

\begin{Example}
(1) If $m=4, n=2$ and $\cc=(1,2,2,2,1,1)$, then $s=3$ and $\conv(\Dc(K_{4,2},\mathfrak{c}))$ is nearly Gorenstein.

(2) If $m=3, n=3$ and $\cc=(3,3,1,1,1,1)$, then $s=2$ and $\conv(\Dc(K_{3,3},\mathfrak{c}))$ is nearly Gorenstein.

(3) If $m=2, n=2$ and $\cc=(3,2,2,2)$, then $s=1$ and $\conv(\Dc(K_{2,2},\mathfrak{c}))$ is nearly Gorenstein.
\end{Example}

Let $(a,\mathfrak{c}) = (a, c_1, \ldots, c_{n})\in (\ZZ_{>0})^{n+1}$ be a sequence for which $$a > c_1 \geq \cdots \geq c_{n} \geq 2, \quad a \geq n+1, \quad a < c_1+\cdots + c_{n}.$$  The convex polytope $\Qc^\sharp_{(a; \mathfrak{c})} \subset \RR^{n}$ which is defined by the system of linear inequalities 

$\bullet$ $0\leq x_i\leq c_i, \, \, \, \, \, i=1\ldots, n$;

$\bullet$ $x_1+\cdots +x_{n}\leq a$

\noindent
is called a convex polytope {\em of Veronese type}.  It then follows from \cite[Section $3$]{HSF_6} together with Corollary \ref{final} that 

\begin{Corollary}
Every nearly Gorenstein polytope of Veronese type is level.  
\end{Corollary}

\section*{Statements and Declarations}
The authors have no Conflict of interest to declare that are relevant to the content of this article.

\section*{Data availability}
Data sharing does not apply to this article as no new data were created or analyzed in this study.

\end{document}